\documentclass[a4paper,11pt]{amsart}
\usepackage[T1]{fontenc}
\usepackage[english]{babel}
\usepackage[cp1252]{inputenc}
\usepackage{amsthm}
\usepackage{amsmath}
\usepackage{amsfonts}
\usepackage{tikz-cd}
\usepackage{pgfplots}
 \usepackage{microtype}
\usepackage[margin=0.8in]{geometry}

\usepackage{xcolor}
 \pgfplotsset{compat=1.18}
\usepgfplotslibrary{fillbetween}

\usepackage{amssymb}
\usepackage{hyperref}
\usepackage{color}
\usepackage{caption}
\usepackage{indentfirst}
\usepackage{amssymb}
\usepackage{eufrak}
\usepackage{mathrsfs}
\usepackage{xypic}

 \usepackage{booktabs,array,graphicx,longtable}
\usepackage{array,booktabs,tabularx}

\usepackage{booktabs}
\usepackage{listings}
\theoremstyle{plain}   

\usepackage{booktabs}
\usepackage{graphicx}
\usepackage{array}

\definecolor{amoebadarkblue}{RGB}{0,63,127}
\definecolor{contourred}{RGB}{224,0,0}

\newif\ifhighdensityplot
\highdensityplotfalse

\definecolor{ContourColor}{RGB}{100,0,10}

\newcommand{\calA}{\mathscr A}
\newcommand{\calC}{\mathcal C}

\newcommand{\R}{\mathbb R}
\newcommand{\C}{\mathbb C}

\newcommand{\Z}{\mathbb Z}

\newcommand{\Log}{\mathrm{Log}\,}

 \newcommand{\Rea}{\operatorname{Re}}
 \newcommand{\Ima}{\operatorname{Im}}

\newcommand{\PP}{\mathbb P}

\newcommand{\di}{\displaystyle}

\newcommand{\T}{(\mathbb C^*)}

\newcommand{\Crit}{\operatorname{Crit}}

\newcommand{\Area}{\operatorname{Area}}

\newcommand{\rank}{\operatorname{rank}}

\newcommand{\supp}{\operatorname{supp}}
\newcommand{\CA}{\mathcal C\mathcal A}

\newcommand{\Newt}{\operatorname{Newt}}

\usepackage{array}
\usepackage{tabularx}

 \usepackage{caption}
 
\newcommand{\Sing}{\operatorname{Sing}}

\newcommand\restr[2]{{
  \left.\kern-\nulldelimiterspace 
  #1 
  \right|_{#2} 
  }}

\newtheorem{remark}{Remark}[section]
\newtheorem*{mtheorem*}{Main Theorem}
\newtheorem{theorem}{Theorem}[section]

\newtheorem{proposition}{Proposition}[section]
\newtheorem{corollary}{Corollary}[section]
\newtheorem{lemma}{Lemma}[section]

\begin{document}
\title{Singularities of Amoeba Contours}
\author{Mounir Nisse}

\address{Mounir Nisse\\
Department of Mathematics, Xiamen University Malaysia, Jalan Sunsuria, Bandar Sunsuria, 43900, Sepang, Selangor, Malaysia.
}
\email{mounir.nisse@gmail.com, mounir.nisse@xmu.edu.my}

\date{}
 
\thanks{This research is supported in part by Xiamen University Malaysia Research Fund (Grant no. XMUMRF/ 2024-C5/IMAT/0013).}

 

\maketitle

\begin{abstract}
We give explicit real-algebraic equations for computing the singularities of amoeba contours, separating degeneracies of the logarithmic critical locus from coincidences of distinct critical lifts.  For plane curves, this yields practical systems detecting nodes, cusps, and multiple branches.  We also show that maximal sparsity does not maximize contour singularities, even for lattice triangles and parallelograms.
\end{abstract}

\section*{Introduction}

Amoebas connect complex algebraic geometry with real and tropical geometry by sending a subvariety of $(\mathbb C^*)^n$ to its image under the logarithmic map.  Their geometry reflects both the equations of the original variety and the combinatorics of its Newton polytope.  Since the early developments of the subject, amoebas and their contours have played an important role in the study of real algebraic curves, tropical limits, and decompositions of complex hypersurfaces \cite{Mikhalkin00,Mikhalkin04}.  The logarithmic Gauss map gives a natural description of the critical locus of the logarithmic map and provides a direct link between the geometry of the variety and the geometry of its amoeba \cite{MadaniNisse13}.

The contour of an amoeba is the set of critical values of the logarithmic map restricted to the corresponding algebraic hypersurface.  Even when the hypersurface is smooth, its contour may have singularities.  These singularities have two different sources.  A singular value may come from a critical point at which the critical locus is singular or at which its logarithmic image fails to be regular.  It may also come from two or more distinct critical points having the same logarithmic image.  The second situation produces self-intersections or multiple branches, whereas the first may produce cusp-type behavior.  This distinction is closely related to the classical study of singularities of smooth mappings \cite{Whitney55,Hirsch76}, but the logarithmic setting requires equations adapted to the algebraic torus and to the logarithmic Gauss map.

A principal motivation for this work comes from the study of amoeba
contours initiated by Lang, Shapiro, and Shustin
\cite{LangShapiroShustin21}.  Their paper investigates the maximal number
of intersections between the contour of a plane and hypersurface amoeba and a real line,
introduces the corresponding real degree, and establishes bounds in terms
of the Newton polygon.  The present work is motivated by their results and
by the questions raised there.  It shifts the emphasis from intersections
with auxiliary lines to the intrinsic singularities of the contour and
develops bounds for the number of cusps and for transverse
$s$-fold nodes.

The first main result of this paper gives explicit polynomial equations for the logarithmic critical locus of a smooth hypersurface in $(\mathbb C^*)^n$.  We then derive Jacobian and determinant conditions that detect singular points of the critical locus and points where the logarithmic map restricted to that locus loses rank.  These equations describe the source-degenerate part of the singular contour.  To detect the second source of singularities, we construct a multiple-lift system consisting of two copies of the critical equations together with equations imposing equality of the logarithmic radii.  The diagonal, which corresponds to using the same critical point twice, is removed by an exact off-diagonal condition.

The central theorem combines these systems and gives a precise computation of the singular contour.  Under natural local properness and finite-fiber assumptions, every singular contour point comes either from a source-degenerate critical lift or from distinct critical lifts with the same logarithmic value.  Conversely, after the local contour branches have been checked, the positive real projection of these systems gives exactly the singular locus.  An important part of the result is the distinction between the actual real projection and the zero set of an elimination ideal.  Ordinary elimination gives the Zariski closure of a projection and may introduce additional points.  The exact real inequalities, the logarithmic Gauss charts, the positivity of the radius variables, and the off-diagonal condition must therefore be retained.  The required real-algebraic tools are standard in real algebraic geometry \cite{BochnakCosteRoy98}, but their use here produces a concrete method adapted specifically to singular amoeba contours.

For plane curves, all the equations become explicit in four real coordinates.  We obtain three systems: a singular-Jacobian system for singular points of the critical curve, an immersion-degeneracy system for critical points of its logarithmic parametrization, and a two-lift system for distinct critical points with the same logarithmic image.  These systems give a direct procedure for computing the singular values and for separating cusp-type values from multiple-branch values.

The method is applied in detail to the smooth curve $z^2w^2+9w+z+3zw+1=0$, whose amoeba contour has eleven singular logarithmic values. Tangent-determinant and higher-jet calculations identify six transverse double images. At four further values, the logarithmic critical locus has a real branch together with a conjugate pair of nonreal branches. After quotienting by complex conjugation, the nonreal pair determines a single logarithmic half-branch that is exactly tangent to the immersed logarithmic image of the real branch. The remaining singular value arises from a degenerate nonreal critical lift. This example demonstrates that the singularities of an amoeba contour cannot be determined solely from the real part of the curve: nonreal critical points contribute essentially to the complete contour.
  
The final result concerns the relation between sparsity and contour singularities.  Let $\Delta=\operatorname{conv}\{(0,0),(3,0),(0,1)\}$.  Every polynomial supported exactly on the vertices of $\Delta$ is reduced, by monomial changes of coordinates, to a line; hence its contour has no finite singular point.  However, the polynomial $w-1+z-z^3$ has Newton polygon $\Delta$ and its contour has an ordinary transverse node at $(0,0)$.  Therefore, the vertex-supported subfamily does not attain the largest possible number of isolated contour singularities among polynomials with Newton polygon $\Delta$.  The same phenomenon occurs for parallelogram Newton polygons: adding nonvertex monomials can create contour singularities that are absent in the corresponding vertex-supported family.  Thus maximal sparsity does not, in general, maximize the number of contour singularities, either for lattice triangles or for parallelograms.

These results give a complete set of equations for the two basic sources of singular contour points, clarify the role of exact real projection, and provide explicit examples in which nonreal critical lifts and nonvertex monomials change the singular geometry of the contour.  They also show that the Newton polygon alone does not determine how singular the contour can be: the choice of monomials and coefficients remains essential.

\noindent{\it Acknowledgements.} The author is deeply grateful to Boris Shapiro for kindly sharing the joint paper with Lionel Lang and Eugenii Shustin \cite{LangShapiroShustin21}; the present work was inspired by the challenging question raised in that paper. The author also thanks Akiyoshi Sannai and AI Math Referee for pointing out a missing singularity in the count of singularities in the example of Section 6 in an earlier version. Finally, the author gratefully acknowledges the use of Codex AI, based on GPT-6, which employed Python, NumPy, SciPy, Matplotlib, Pillow, and \LaTeX{} with TikZ and PGFPlots for the exact computations, numerical experiments, and figures presented in this work. 


\section{Equations and Jacobian Criteria for Singularities of Amoeba Contours}
 
Let $f\in\C[z_1^{\pm1},\ldots,z_n^{\pm1}]$ be a Laurent polynomial and assume that the hypersurface $V=\{z\in(\C^\ast)^n\mid f(z)=0\}$ is smooth. Its Newton polytope is denoted by  $\Delta=\Newt(f)$. We consider the logarithmic map $\Log:(\C^\ast)^n\to\R^n$, defined by $\Log(z_1,\ldots,z_n)=(\log|z_1|,\ldots,\log|z_n|)$, and its restriction $L=\Log|_V$. The critical locus of $L$ is denoted by $\Sigma=\Crit(L)$, and the contour of the amoeba to $V$ is denoted by $\CA_V=L(\Sigma)$.

For $1\leq j\leq n$, set $D_j=z_j\partial_{z_j}$ and $F_j=D_jf=z_jf_{z_j}$. Since $V$ is smooth and every coordinate $z_j$ is nonzero on $(\C^\ast)^n$, the functions $F_1,\ldots,F_n$ do not vanish simultaneously on $V$. The logarithmic Gauss map is therefore the holomorphic map
$$
\gamma_{\log}:V\longrightarrow\PP^{n-1},\qquad
\gamma_{\log}(z)=[F_1(z):\cdots:F_n(z)].
$$
For hypersurfaces, the logarithmic Gauss map records the projective normal direction after passing to local holomorphic logarithmic coordinates. The relation between this map and the logarithmic critical locus is due to Mikhalkin in the hypersurface case \cite{Mikhalkin04}.

\begin{theorem}[Polynomial equations for the critical locus]
Let $V\subset(\C^\ast)^n$ be a smooth hypersurface defined by $f$. Then
$
\Sigma=\Crit(\Log|_V)=\gamma_{\log}^{-1}(\R\PP^{n-1}),
$
which is equivalent to say that a point $q\in V$ lies in $\Sigma$ if and only if
$$
F_j(q)\overline{F_k(q)}-\overline{F_j(q)}F_k(q)=0
$$
for every $1\leq j<k\leq n$. In real coordinates $z_j=x_j+iy_j$, the critical locus is therefore a real algebraic subset of $(\R^2\setminus\{0\})^n$ defined by
$
\Rea f=0,\,  \Ima f=0,\, 
\Ima\!\left(F_j\overline{F_k}\right)=0
$
for all $1\leq j<k\leq n$.

On the affine chart $U_\ell=\{q\in V\mid F_\ell(q)\neq0\}$, it is enough to impose the $n-1$ equations
$$
\rho_j^{(\ell)}=\Ima\!\left(F_j\overline{F_\ell}\right)=0,
\qquad j\neq\ell.
$$
Thus
$$
\Sigma\cap U_\ell
=
\left\{
q\in U_\ell
\ \middle|\ 
\rho_j^{(\ell)}(q)=0\text{ for every }j\neq\ell
\right\}.
$$
\end{theorem}
\begin{proof}
Fix $q\in V$ and choose a local holomorphic logarithm $w_j=\log z_j$ near $q$. In the coordinates $w=(w_1,\ldots,w_n)$, the hypersurface $V$ is locally defined by $\widetilde f(w)=f(e^{w_1},\ldots,e^{w_n})=0$, and the chain rule gives $\partial\widetilde f/\partial w_j=F_j$. Hence the complex tangent hyperplane to the logarithmic transform of $V$ at $q$ is
$$\di
H_q=
\{
v\in\C^n
 \mid
\sum_{j=1}^nF_j(q)v_j=0
\}.
$$
The differential of $\Log|_V$ is identified with the restriction of the real-part map $\Rea:\C^n\to\R^n$ to $H_q$.

Let $a=(a_1,\ldots,a_n)\in\C^n\setminus\{0\}$ and $H_a=\{v\in\C^n\mid \sum a_jv_j=0\}$. The map $\Rea|_{H_a}$ is not surjective if and only if $[a_1:\cdots:a_n]\in\R\PP^{n-1}$. Indeed, if $a=\lambda r$ with $\lambda\in\C^\ast$ and $r\in\R^n\setminus\{0\}$, then $\Rea(H_a)=r^\perp$, so the rank is $n-1$. Conversely, if $\Rea(H_a)$ is a proper subspace of $\R^n$, choose $r\in\R^n\setminus\{0\}$ annihilating it. Since $H_a$ is invariant under multiplication by $i$, the identities $r\cdot\Rea(v)=0$ and $r\cdot\Rea(iv)=0$ imply $r\cdot v=0$ for all $v\in H_a$. Thus the complex linear forms $a\cdot v$ and $r\cdot v$ have the same kernel and are proportional. Therefore $[a]\in\R\PP^{n-1}$.
Applying this linear-algebra statement to $a=(F_1(q),\ldots,F_n(q))$ proves
$$
q\in\Crit(\Log|_V)
\quad\Longleftrightarrow\quad
\gamma_{\log}(q)\in\R\PP^{n-1}.
$$
 
A projective point $[F_1:\cdots:F_n]$ is real if and only if all nonzero coordinates have one common complex phase. This is equivalent to $F_j\overline{F_k}-\overline{F_j}F_k=0$ for all $j<k$. On the chart $F_\ell\neq0$, it is enough to require $F_j/F_\ell\in\R$ for every $j\neq\ell$, and this is exactly the condition $\Ima(F_j\overline{F_\ell})=0$. The stated real equations follow after writing $z_j=x_j+iy_j$ and replacing every Laurent monomial by an ordinary real rational expression; multiplication by a sufficiently large monomial clears all denominators without changing the zero set inside $(\C^\ast)^n$.
\end{proof}

\begin{remark}
The equations $\rho_j^{(\ell)}=0$ depend on the chosen chart, but their common zero set does not. For symbolic computation it is usually preferable to cover $V$ by the charts $F_\ell\neq0$, because the full family $\Ima(F_j\overline{F_k})=0$ contains many algebraic dependencies.
\end{remark}

\subsection{The Jacobian criterion for the restricted logarithmic map}

Fix an index $\ell$ and work on $U_\ell$. Write $z_j=x_j+iy_j$ and let $X=(x_1,y_1,\ldots,x_n,y_n)$ be the real coordinates on $(\C^\ast)^n$. Define the real-valued functions
$
g_1=\Rea f,\,  g_2=\Ima f,
$
and choose an ordering $j_1,\ldots,j_{n-1}$ of $\{1,\ldots,n\}\setminus\{\ell\}$. Put
$
g_{2+r}=\rho_{j_r}^{(\ell)}
=\Ima\!\left(F_{j_r}\overline{F_\ell}\right),
\,  1\leq r\leq n-1.
$
Let $G_\ell=(g_1,\ldots,g_{n+1}):U_\ell\to\R^{n+1}$. Then $\Sigma\cap U_\ell=G_\ell^{-1}(0)$.
The differential of the ambient logarithmic map is the $n\times2n$ matrix
$$
D\Log(X)=
\begin{pmatrix}
\dfrac{x_1}{x_1^2+y_1^2} & \dfrac{y_1}{x_1^2+y_1^2} & 0 & 0 & \cdots & 0 & 0\\
0 & 0 & \dfrac{x_2}{x_2^2+y_2^2} & \dfrac{y_2}{x_2^2+y_2^2} & \cdots & 0 & 0\\
\vdots & \vdots & \vdots & \vdots & \ddots & \vdots & \vdots\\
0 & 0 & 0 & 0 & \cdots & \dfrac{x_n}{x_n^2+y_n^2} & \dfrac{y_n}{x_n^2+y_n^2}
\end{pmatrix}.
$$

\begin{theorem}
Let $q\in\Sigma\cap U_\ell$ and suppose that
$
\rank DG_\ell(q)=n+1.
$
Then $\Sigma$ is a smooth real analytic manifold of dimension $n-1$ near $q$. Let $h=L|_\Sigma$. The differential $dh_q:T_q\Sigma\to\R^n$ has rank smaller than $n-1$ if and only if
$\di
\rank
\begin{pmatrix}
DG_\ell(q)\\
D\Log(q)
\end{pmatrix}
<2n.
$
Since the displayed matrix is square of size $2n$, this is equivalent to
$\di
\det
\begin{pmatrix}
DG_\ell(q)\\
D\Log(q)
\end{pmatrix}
=0.
$
Thus, on the smooth part of $\Sigma\cap U_\ell$, the degeneracy locus of the restricted logarithmic map is cut out by the equations $G_\ell=0$ together with the single Jacobian determinant equation
$$
J_\ell=
\det
\begin{pmatrix}
DG_\ell\\
D\Log
\end{pmatrix}
=0.
$$
After multiplication by a power of $\prod_{j=1}^n(x_j^2+y_j^2)$, the equation $J_\ell=0$ becomes polynomial in the real coordinates.
\end{theorem}
\begin{proof}
The hypothesis $\rank DG_\ell(q)=n+1$ and the regular-value theorem imply that $G_\ell^{-1}(0)$ is a smooth submanifold near $q$ of codimension $n+1$, hence of dimension $n-1$ \cite[Chapter~1, Theorem~3.2]{Hirsch76}. Its tangent space is $T_q\Sigma=\ker DG_\ell(q)$.

Let $A=DG_\ell(q):\R^{2n}\to\R^{n+1}$ and $B=D\Log(q):\R^{2n}\to\R^n$. Since $A$ is surjective, the rank of the restriction $B|_{\ker A}$ satisfies
$$
\rank
\begin{pmatrix}
A\\
B
\end{pmatrix}
=
\rank A+\rank(B|_{\ker A}).
$$
To verify this identity, choose a direct-sum decomposition $\R^{2n}=\ker A\oplus E$ such that $A|_E:E\to\R^{n+1}$ is an isomorphism. The contribution of $E$ to the image of the map $v\mapsto(Av,Bv)$ has dimension $n+1$, while the additional contribution from $\ker A$ is precisely $\{0\}\oplus B(\ker A)$ and has dimension $\rank(B|_{\ker A})$.
Since $\rank A=n+1$, one obtains
$$
\rank(dh_q)
=
\rank(B|_{\ker A})
=
\rank
\begin{pmatrix}
A\\
B
\end{pmatrix}
-(n+1).
$$
Thus $\rank(dh_q)<n-1$ if and only if the augmented matrix has rank less than $(n+1)+(n-1)=2n$. Since it has exactly $2n$ rows and $2n$ columns, this is equivalent to the vanishing of its determinant.

The entries of $DG_\ell$ are polynomial after clearing Laurent denominators, while the entries of $D\Log$ have denominators $x_j^2+y_j^2$. These denominators never vanish on $(\C^\ast)^n$. Multiplying $J_\ell$ by a sufficiently large power of $\prod_j(x_j^2+y_j^2)$ therefore produces a polynomial equation with the same zero set in the algebraic torus.
\end{proof}

\begin{remark}
The condition $\rank DG_\ell(q)=n+1$ is the regularity condition for the critical locus itself. If it fails, then $q$ is a singular point of the real critical set in this chart. Such points must be included among the degenerate critical lifts even before one studies the rank of $h=L|_\Sigma$.
\end{remark}

\subsection{Every singular contour point comes from degeneration or multiple lifts}
We call $p\in\CA_V$ a regular contour point if there is a neighborhood $W$ of $p$ in $\R^n$ such that $\CA_V\cap W$ is a smooth embedded real hypersurface. Otherwise $p$ is called a singular contour point.

For a critical lift $q\in\Sigma$ with $L(q)=p$, say that $q$ is nondegenerate for the contour map if $\Sigma$ is smooth near $q$ and $\rank d(L|_\Sigma)_q=n-1$. In that case $L|_\Sigma$ is an immersion at $q$.

\begin{theorem}
Let $p\in\CA_V$. Assume that there is a neighborhood $W$ of $p$ and an open neighborhood $\Omega$ of the compact set $\Sigma\cap L^{-1}(p)$ such that the restriction
$$
L:\Sigma\cap\Omega\longrightarrow W
$$
is proper. Suppose also that $\Sigma\cap L^{-1}(p)$ is finite. If $p$ is a singular point of the contour, then either there exists a lift $q\in\Sigma\cap L^{-1}(p)$ at which $\Sigma$ is singular or $\rank d(L|_\Sigma)_q<n-1$, or there exist two distinct critical lifts $q_1,q_2\in\Sigma\cap L^{-1}(p)$.
Equivalently, if $p$ has exactly one critical lift $q$, if $\Sigma$ is smooth near $q$, and if $\rank d(L|_\Sigma)_q=n-1$, then $p$ is a regular contour point.
\end{theorem}


\begin{proof}
Assume that $p$ has exactly one critical lift $q$, that $\Sigma$ is smooth near $q$, and that $\rank d(L|_\Sigma)_q=n-1$. Since $\dim_\R\Sigma=n-1$, the differential of $h=L|_\Sigma$ is injective at $q$. By the constant-rank theorem, there is a neighborhood $U$ of $q$ in $\Sigma$ such that $h|_U$ is an immersion and, after shrinking $U$, an embedding onto a smooth embedded $(n-1)$-dimensional submanifold of $\R^n$ \cite[Chapter~2, Theorem~5.1]{Hirsch76}. Denote this image by $M=h(U)$.

It remains to exclude the possibility that points of $\Sigma\setminus U$ have logarithmic images arbitrarily close to $p$ and thereby create additional local branches. Suppose that no neighborhood of $p$ has this exclusion property. Then there are points $q_\nu\in(\Sigma\cap\Omega)\setminus U$ such that $L(q_\nu)\to p$. Properness of $L:\Sigma\cap\Omega\to W$ implies that, after passing to a subsequence, $q_\nu$ converges to a point $q_\infty\in\Sigma\cap\Omega$ with $L(q_\infty)=p$. The fiber over $p$ consists only of $q$, so $q_\infty=q$. For $\nu$ sufficiently large, this gives $q_\nu\in U$, contradicting the choice of $q_\nu$.

Hence, after shrinking $W$ around $p$, one has $\CA_V\cap W=L(U)\cap W=M\cap W$. Thus the contour is a smooth embedded hypersurface near $p$, so $p$ is regular. The contrapositive proves the theorem.
\end{proof}

\begin{corollary}
On each chart $U_\ell$, define the real algebraic degeneracy set
$$
\mathcal D_\ell=
\left\{
q\in U_\ell
\ \middle|\ 
G_\ell(q)=0,\ 
\rank DG_\ell(q)<n+1
\text{ or }
J_\ell(q)=0
\right\}.
$$
Define also the double-lift incidence set
$$
\mathcal M=
\left\{
(q,q')\in\Sigma\times\Sigma
\ \middle|\ 
q\neq q',\ \Log(q)=\Log(q')
\right\}.
$$
Under the local properness and finiteness hypotheses of the theorem, every singular contour point belongs to
$$
\bigcup_{\ell=1}^n\Log(\mathcal D_\ell)
\ \cup\ 
\Log\!\left(\operatorname{pr}_1(\mathcal M)\right).
$$
Thus the singular locus of the contour is contained in the union of the logarithmic images of the Jacobian-degeneracy locus and the multiple-lift locus.
\end{corollary}
\begin{proof}
Let $p$ be singular. The preceding theorem shows that either one critical lift is degenerate or there are two distinct critical lifts. In the first case, choose a chart $U_\ell$ containing the lift. If the critical locus is singular, then $\rank DG_\ell<n+1$. If the critical locus is smooth but the restricted logarithmic differential has rank smaller than $n-1$, then $J_\ell=0$ by the Jacobian criterion. Hence the lift lies in $\mathcal D_\ell$. In the second case, the pair of distinct lifts belongs to $\mathcal M$. Taking logarithmic images gives the asserted inclusion.
\end{proof}

The preceding results provide a finite chartwise procedure. On each chart $F_\ell\neq0$, one writes the real equations $G_\ell=0$ for the critical locus. One then computes all $(n+1)\times(n+1)$ minors of $DG_\ell$ to detect singular points of the critical locus and computes the determinant $J_\ell$ to detect rank loss of the restricted logarithmic map. These equations define the source-degeneracy locus.

The multiple-lift locus is described by two copies of the critical equations together with equality of logarithmic moduli. If $q=(z_1,\ldots,z_n)$ and $q'=(z_1',\ldots,z_n')$, then $\Log(q)=\Log(q')$ is equivalent to
$$
|z_j|^2=|z_j'|^2
\qquad
\text{for }1\leq j\leq n.
$$
In real coordinates this becomes $x_j^2+y_j^2=(x_j')^2+(y_j')^2$. The diagonal $q=q'$ must be removed. Algebraically, this removal is implemented by saturation with respect to the ideal generated by the coordinate differences $x_j-x_j'$ and $y_j-y_j'$. The projection of this incidence variety to the first copy of the source, followed by $\Log$, gives the multiple-lift candidate set.

The equations identify a rigorous superset of the singular contour locus. Equality need not hold without additional hypotheses. A degenerate lift can still have a smooth image, and two distinct lifts can parametrize the same smooth local image branch. Determining the exact analytic type requires comparison of higher jets of the restricted logarithmic map along the relevant branches. The theorem above asserts only the necessary dichotomy that every singular image point must come from source degeneration or from more than one critical lift.

\subsection*{Conclusion}
For a smooth hypersurface, the computation of possible singular points of the amoeba contour begins with the logarithmic first derivatives $F_j=D_jf$. The critical locus is obtained by imposing projective reality of $[F_1:\cdots:F_n]$. On a chart $F_\ell\neq0$, this gives $n-1$ explicit real equations. The singular points of the critical locus are detected by the rank of $DG_\ell$, while the failure of the restricted logarithmic map to be immersive is detected by the square augmented Jacobian determinant $J_\ell$. Singular contour points not detected in this way must arise from distinct critical lifts with equal logarithmic modulus. This produces an  algebraic-elimination  plan for finding all candidate singularities of the contour.


\section{Degenerate Critical Lifts, and Singularities of Amoeba Contours}

Let $f\in\C[z_1^{\pm1},\ldots,z_n^{\pm1}]$ be a Laurent polynomial and suppose that $V=\{z\in(\C^\ast)^n\mid f(z)=0\}$ is smooth. Let $L=\Log|_V:V\longrightarrow\R^n$ and let $\Sigma=\Crit(L)$. The contour of the amoeba is $\CA_V=L(\Sigma)$. For $1\leq j\leq n$, put $D_j=z_j\partial_{z_j}$ and $F_j=D_jf$. The logarithmic Gauss map is
$$
\gamma_{\log}:V\longrightarrow\PP^{n-1},\qquad
\gamma_{\log}(z)=[F_1(z):\cdots:F_n(z)].
$$
The critical-locus identity $\Sigma=\gamma_{\log}^{-1}(\R\PP^{n-1})$ is the logarithmic Gauss characterization of the critical points of the logarithmic map; it appears for hypersurfaces in Mikhalkin's work and is generalized to arbitrary codimension by Madani and Nisse \cite[Lemma~3]{Mikhalkin00}\cite[Theorem~3.2]{MadaniNisse13}.

Write $z_j=x_j+iy_j$ and let $X=(x_1,y_1,\ldots,x_n,y_n)\in\R^{2n}$. Set $u=\Rea f$ and $v=\Ima f$. On the chart $U_\ell=\{F_\ell\neq0\}$, choose an ordering $j_1,\ldots,j_{n-1}$ of $\{1,\ldots,n\}\setminus\{\ell\}$ and define
$$
\rho_r^{(\ell)}=\Ima\!\left(F_{j_r}\overline{F_\ell}\right),\qquad 1\leq r\leq n-1.
$$
Then
$$
G_\ell=\left(u,v,\rho_1^{(\ell)},\ldots,\rho_{n-1}^{(\ell)}\right):U_\ell\longrightarrow\R^{n+1}
$$
satisfies $\Sigma\cap U_\ell=G_\ell^{-1}(0)$. The equations above are real algebraic after clearing Laurent denominators.

The differential of the ambient logarithmic map is the $n\times2n$ matrix
$$
D\Log(X)=
\begin{pmatrix}
\dfrac{x_1}{x_1^2+y_1^2} & \dfrac{y_1}{x_1^2+y_1^2} & 0 & 0 & \cdots & 0 & 0\\
0 & 0 & \dfrac{x_2}{x_2^2+y_2^2} & \dfrac{y_2}{x_2^2+y_2^2} & \cdots & 0 & 0\\
\vdots & \vdots & \vdots & \vdots & \ddots & \vdots & \vdots\\
0 & 0 & 0 & 0 & \cdots & \dfrac{x_n}{x_n^2+y_n^2} & \dfrac{y_n}{x_n^2+y_n^2}
\end{pmatrix}.
$$

\subsection{Degeneration of the restricted logarithmic map}

Let $h=L|_{\Sigma}$. At a point at which $\Sigma$ is smooth of dimension $n-1$, the contour map is locally regular precisely when $dh$ has rank $n-1$.

\begin{proposition}
Let $q\in\Sigma\cap U_\ell$ and assume $\rank DG_\ell(q)=n+1$. Then $\Sigma$ is a smooth real analytic manifold of dimension $n-1$ near $q$, and $T_q\Sigma=\ker DG_\ell(q)$. Moreover,
$$
\rank dh_q=
\rank
\begin{pmatrix}
DG_\ell(q)\\
D\Log(q)
\end{pmatrix}
-(n+1).
$$
Consequently, $\rank dh_q<n-1$ if and only if
$$
\rank
\begin{pmatrix}
DG_\ell(q)\\
D\Log(q)
\end{pmatrix}<2n.
$$
Since the augmented matrix has $2n+1$ rows and $2n$ columns, the latter condition is equivalent to the vanishing of all its $2n\times2n$ minors.
\end{proposition}

\begin{proof}
The regular-value theorem gives the smoothness and dimension of $G_\ell^{-1}(0)$ because $DG_\ell(q)$ has full row rank $n+1$ \cite[Chapter~1, Section~3]{Hirsch76}. Put $A=DG_\ell(q)$ and $B=D\Log(q)$. Choose a direct-sum decomposition $\R^{2n}=\ker A\oplus E$ such that $A|_E:E\to\R^{n+1}$ is an isomorphism. For $\Phi=(A,B):\R^{2n}\to\R^{n+1}\oplus\R^n$, the image of $E$ contributes $n+1$ independent dimensions through the first component, while the additional image contributed by $\ker A$ is $\{0\}\oplus B(\ker A)$. Therefore
$$
\rank\Phi=n+1+\rank(B|_{\ker A}).
$$
Since $B|_{\ker A}=dh_q$, the rank identity follows. The inequality $\rank dh_q<n-1$ is then equivalent to $\rank\Phi<2n$. A matrix with $2n$ columns has rank smaller than $2n$ precisely when all maximal $2n\times2n$ minors vanish.
\end{proof}

For each chart $U_\ell$, let $\mathcal D_\ell$ be the union of the singular locus of the critical equations and the degeneracy locus of the restricted logarithmic map. Set-theoretically,
$$
\mathcal D_\ell=
\left\{G_\ell=0,\ \rank DG_\ell<n+1\right\}
\cup
\left\{G_\ell=0,\ \rank DG_\ell=n+1,\ \rank\begin{pmatrix}DG_\ell\\ D\Log\end{pmatrix}<2n\right\}.
$$
For elimination purposes it is preferable to treat the two pieces separately, because the first is defined by the $(n+1)\times(n+1)$ minors of $DG_\ell$, whereas the second is defined by the $2n\times2n$ minors of the augmented matrix after restricting to the open set where at least one $(n+1)\times(n+1)$ minor of $DG_\ell$ is nonzero.

\subsection{The elimination system}

Introduce positive radius variables $R=(R_1,\ldots,R_n)$ with $R_j=x_j^2+y_j^2$. The logarithmic target point corresponding to $R$ is $p(R)=\frac12(\log R_1,\ldots,\log R_n)$. Working with $R$ makes the projection algebraic; the final logarithm is applied only after elimination.

Let $I_{\Sigma,\ell}$ be the real polynomial ideal generated by the cleared-denominator forms of $u,v,\rho_1^{(\ell)},\ldots,\rho_{n-1}^{(\ell)}$. Let $I_{\mathrm{sing},\ell}$ be the ideal generated by $I_{\Sigma,\ell}$ and all $(n+1)\times(n+1)$ minors of $DG_\ell$. Let $I_{\mathrm{imm},\ell}$ be the ideal generated by $I_{\Sigma,\ell}$ and all $2n\times2n$ minors of
$\di
\begin{pmatrix}
DG_\ell\\
D\Log
\end{pmatrix},
$
after clearing the nonvanishing denominators $\prod_j(x_j^2+y_j^2)$. To remain in the chart $F_\ell\neq0$, these ideals must be saturated by $|F_\ell|^2=F_\ell\overline{F_\ell}$. Adjoin the radius equations
$$
R_j-x_j^2-y_j^2=0,
\qquad 1\leq j\leq n.
$$
Eliminating the source variables $X$ from
$
I_{\mathrm{sing},\ell}+\langle R_j-x_j^2-y_j^2\rangle
$
and from
$
I_{\mathrm{imm},\ell}+\langle R_j-x_j^2-y_j^2\rangle
$
gives ideals in $\R[R_1,\ldots,R_n]$. Their positive real zero sets contain the radius images of all degenerate critical lifts.

For multiple lifts, introduce a second source point $X'=(x_1',y_1',\ldots,x_n',y_n')$ and chart indices $\ell,m$. Let $I_{\mathrm{mult},\ell m}$ be generated by $I_{\Sigma,\ell}(X)$, $I_{\Sigma,m}(X')$, and
$
x_j^2+y_j^2-(x_j')^2-(y_j')^2,
\,  1\leq j\leq n.
$
The diagonal must be removed. Set
$$
\delta(X,X')=\sum_{j=1}^n\left((x_j-x_j')^2+(y_j-y_j')^2\right).
$$
Over the real numbers, $\delta=0$ is equivalent to $X=X'$. The off-diagonal incidence ideal is therefore the saturation
$
I_{\mathrm{mult},\ell m}^{\mathrm{off}}=I_{\mathrm{mult},\ell m}:\langle\delta\rangle^\infty.
$
After adjoining $R_j-x_j^2-y_j^2=0$ and eliminating both $X$ and $X'$, one obtains a radius-space ideal whose positive real zero set is the closure of the radii admitting two distinct critical lifts.


\section{The Elimination Description for Singular Amoeba Contours}

Let $V\subset(\C^\ast)^n$ be a smooth algebraic hypersurface and let $L=\Log|_V$. Write $\Sigma=\Crit(L)$ and $\CA_V=L(\Sigma)$. For a critical point $q=(z_1,\ldots,z_n)\in\Sigma$, write $z_j=x_j+iy_j$ and define its radius vector by
$
\mathcal R(q)=(|z_1|^2,\ldots,|z_n|^2)=(x_1^2+y_1^2,\ldots,x_n^2+y_n^2)\in\R_{>0}^n.
$
Then
$
L(q)=\frac12(\log R_1,\ldots,\log R_n)
$
whenever $\mathcal R(q)=(R_1,\ldots,R_n)$.

For each logarithmic Gauss chart $U_\ell$, let $I_{\mathrm{sing},\ell}$ denote the real polynomial ideal defining the points of the critical locus at which the chartwise critical equations fail to have maximal rank, and let $I_{\mathrm{imm},\ell}$ denote the real polynomial ideal defining the smooth critical points at which the differential of $L|_\Sigma$ has rank strictly smaller than $n-1$. The ideals are understood to be saturated by the chart equation ensuring $F_\ell\neq0$. Let $I_{\mathrm{mult},\ell m}^{\mathrm{off}}$ denote the saturated two-point incidence ideal defining distinct critical points in the charts $U_\ell$ and $U_m$ having equal radius vectors.

The phrase ``positive real radius set obtained from an ideal'' will be used in two different senses. Algebraic elimination produces the real points of a Zariski-closed projection and may therefore add points that are not actual projections of real solutions. Exact semialgebraic projection means the literal set of positive radius vectors for which the defining real equations, inequalities, and chart conditions admit a real source solution. This distinction is essential for the equality statement.

\begin{theorem}
Let $E_{\mathrm{deg}}\subset\R_{>0}^n$ be the union over $\ell$ of the positive real radius sets obtained from $I_{\mathrm{sing},\ell}$ and $I_{\mathrm{imm},\ell}$. Let $E_{\mathrm{mult}}\subset\R_{>0}^n$ be the union over $\ell,m$ of the positive real projections of the off-diagonal incidence systems $I_{\mathrm{mult},\ell m}^{\mathrm{off}}$. Under the local properness and finite-fiber hypotheses of the preceding theorem,
$$
\Sing(\CA_V)
\subset
\left\{
\frac12(\log R_1,\ldots,\log R_n)
\ \middle|\
R\in E_{\mathrm{deg}}\cup E_{\mathrm{mult}}
\right\}.
$$
If every nondegenerate multiple fiber produces at least two distinct local image hypersurface germs, and every point of $E_{\mathrm{deg}}$ retained after projection has a singular image germ, then equality holds after replacing the algebraic elimination closures by the corresponding exact semialgebraic projections.
\end{theorem}

\begin{proof}
Define the logarithmic radius map
$
\lambda:\R_{>0}^n\to\R^n
$
by
$$
\lambda(R_1,\ldots,R_n)
=
\frac12(\log R_1,\ldots,\log R_n).
$$
By construction,
$
L=\lambda\circ\mathcal R
$
on $(\C^\ast)^n$.
Let $p\in\Sing(\CA_V)$. The local properness and finite-fiber hypotheses allow the preceding source theorem to be applied at $p$. Hence at least one of the following two geometric alternatives occurs. There is a critical lift $q\in\Sigma\cap L^{-1}(p)$ at which either $\Sigma$ is singular or the restricted differential
$
d(L|_\Sigma)_q
$
has rank strictly smaller than $n-1$. Otherwise there are two distinct critical lifts
$
q,q'\in\Sigma\cap L^{-1}(p).
$
Although these alternatives may occur simultaneously, one of them must occur.

Assume first that there is a degenerate critical lift $q$. Since the logarithmic Gauss charts cover $\Sigma$, there exists an index $\ell$ such that $q\in U_\ell$. If $\Sigma$ is singular at $q$, then the chartwise defining equations of $\Sigma$ vanish at $q$ and their Jacobian has rank smaller than its maximal value. Thus $q$ is a real solution of the system defined by $I_{\mathrm{sing},\ell}$. If instead $\Sigma$ is smooth at $q$ and
$
\rank d(L|_\Sigma)_q<n-1,
$
then the augmented-Jacobian criterion shows that all maximal minors of the augmented matrix formed from the Jacobian of the critical equations and $D\Log$ vanish at $q$. Thus $q$ is a real solution of the system defined by $I_{\mathrm{imm},\ell}$.

Let
$
R=\mathcal R(q).
$
The radius equations
$
R_j-x_j^2-y_j^2=0
$
hold for every $j$. Therefore $R$ belongs to the exact positive real radius projection of the corresponding degenerate-lift system. Any algebraic elimination set used to define $E_{\mathrm{deg}}$ contains this exact projection, because elimination contains the image of every common zero of the original system. Hence
$
R\in E_{\mathrm{deg}}.
$
Since $L(q)=p$, one has
$
p=\lambda(R).
$
Therefore
$
p\in
\left\{
\lambda(R)
\ \middle|\
R\in E_{\mathrm{deg}}
\right\}.
$

Assume now that there are two distinct critical lifts
$
q\neq q'
$
with
$
L(q)=L(q')=p.
$
Choose indices $\ell$ and $m$ such that
$
q\in U_\ell
$
and
$
q'\in U_m.
$
Both points satisfy their respective chartwise critical equations. Equality of logarithmic images gives
$
\log|z_j|=\log|z_j'|
$
for every $j$, and therefore
$
|z_j|^2=|z_j'|^2
$
for every $j$. Thus $(q,q')$ satisfies the two-copy critical equations and the equal-radius equations defining the multiple-lift incidence system.
Since $q\neq q'$, the diagonal-removing polynomial
$$
\delta(q,q')
=
\sum_{j=1}^n
\left(
(x_j-x_j')^2+(y_j-y_j')^2
\right)
$$
is strictly positive. Consequently $(q,q')$ survives saturation by $\delta$, and hence it is a real point of the off-diagonal incidence system
$
I_{\mathrm{mult},\ell m}^{\mathrm{off}}.
$
Let
$
R=\mathcal R(q)=\mathcal R(q').
$
Then $R$ belongs to the exact positive real projection of this incidence system, and therefore also to any algebraic elimination set used to define $E_{\mathrm{mult}}$. Again
$
p=\lambda(R).
$
It follows that
$
p\in
\left\{
\lambda(R)
\ \middle|\
R\in E_{\mathrm{mult}}
\right\}.
$
Combining the two alternatives proves
$$
\Sing(\CA_V)
\subset
\lambda(E_{\mathrm{deg}}\cup E_{\mathrm{mult}}).
$$

It remains to prove the equality assertion. For this purpose the algebraic elimination closures must be replaced by the exact semialgebraic projections. Denote these exact sets by
$
E_{\mathrm{deg}}^{\mathrm{ex}}
$
and
$
E_{\mathrm{mult}}^{\mathrm{ex}}.
$
Thus $R\in E_{\mathrm{deg}}^{\mathrm{ex}}$ precisely when there exists an actual real critical lift $q$ with radius vector $R$ that is singular in $\Sigma$ or at which
$
\rank d(L|_\Sigma)_q<n-1.
$
Likewise, $R\in E_{\mathrm{mult}}^{\mathrm{ex}}$ precisely when there exist two actual distinct real critical lifts $q,q'$ with common radius vector $R$.
The proved inclusion remains valid with the exact sets:
$$
\Sing(\CA_V)
\subset
\lambda
\left(
E_{\mathrm{deg}}^{\mathrm{ex}}
\cup
E_{\mathrm{mult}}^{\mathrm{ex}}
\right).
$$
We prove the reverse inclusion.

Let
$
R\in E_{\mathrm{deg}}^{\mathrm{ex}}
$
and put
$
p=\lambda(R).
$
By definition there exists a degenerate critical lift
$
q\in\Sigma
$
with
$
\mathcal R(q)=R
$
and
$
L(q)=p.
$
The additional hypothesis in the theorem states that every point of the exact degenerate radius projection retained in the computation has a singular image germ. Applied to $R$, this means that the germ of $\CA_V$ at $p$ is not a smooth embedded hypersurface germ. Hence
$
p\in\Sing(\CA_V).
$
Therefore
$
\lambda(E_{\mathrm{deg}}^{\mathrm{ex}})
\subset
\Sing(\CA_V).
$

Now let
$
R\in E_{\mathrm{mult}}^{\mathrm{ex}}
$
and put
$
p=\lambda(R).
$
There exist distinct critical lifts
$
q,q'\in\Sigma
$
such that
$
\mathcal R(q)=\mathcal R(q')=R
$
and therefore
$
L(q)=L(q')=p.
$
If at least one of these lifts is degenerate, then $R$ also belongs to
$
E_{\mathrm{deg}}^{\mathrm{ex}}.
$
The preceding paragraph then yields
$
p\in\Sing(\CA_V).
$

Suppose instead that both lifts are nondegenerate. Then $\Sigma$ is smooth near both $q$ and $q'$, and
$
d(L|_\Sigma)
$
has rank $n-1$ at each point. By the constant-rank theorem, there are neighborhoods
$
U\subset\Sigma
$
of $q$ and
$
U'\subset\Sigma
$
of $q'$
such that
$
L(U)
$
and
$
L(U')
$
are smooth local hypersurface germs through $p$. The additional multiple-fiber hypothesis states that these two local image germs are distinct. The germ of the contour at $p$ therefore contains at least two distinct smooth hypersurface germs.

A union containing two distinct smooth hypersurface germs through the same point cannot itself be a single smooth embedded hypersurface germ. Indeed, suppose for contradiction that the contour germ were represented by a smooth embedded hypersurface $M$ near $p$. Since both $L(U)$ and $L(U')$ are contained in the contour germ, both would be contained in $M$ near $p$. Each has the same dimension $n-1$ as $M$. The inclusion of one smooth $(n-1)$-manifold germ into another smooth $(n-1)$-manifold germ is locally open by the inverse function theorem applied to the inclusion in local coordinates. Hence each of the germs $L(U)$ and $L(U')$ would coincide with the germ of $M$ at $p$. They would consequently coincide with each other, contradicting the assumed distinctness. Thus the contour germ at $p$ is singular, and
$
p\in\Sing(\CA_V).
$
It follows that
$
\lambda(E_{\mathrm{mult}}^{\mathrm{ex}})
\subset
\Sing(\CA_V).
$

Combining the two reverse inclusions gives
$
\lambda
\left(
E_{\mathrm{deg}}^{\mathrm{ex}}
\cup
E_{\mathrm{mult}}^{\mathrm{ex}}
\right)
\subset
\Sing(\CA_V).
$
Together with the previously established opposite inclusion, this yields
$$
\Sing(\CA_V)
=
\left\{
\frac12(\log R_1,\ldots,\log R_n)
\ \middle|\
R\in
E_{\mathrm{deg}}^{\mathrm{ex}}
\cup
E_{\mathrm{mult}}^{\mathrm{ex}}
\right\}.
$$
This is precisely the asserted equality after replacing algebraic elimination closures by exact semialgebraic projections.
\end{proof}


\section{Specialization to plane curves}

Assume now that $n=2$ and write the coordinates as $(z,w)\in(\C^\ast)^2$. Let
$
A=zf_z,
\,
B=wf_w.
$
The logarithmic Gauss map is $\gamma_{\log}=[A:B]$. On the chart $B\neq0$, the critical equation is
$
\rho=\Ima(A\overline B)=0.
$
Writing $z=x+iy$ and $w=s+it$, set
$
u=\Rea f,
\,
v=\Ima f.
$
Then the critical locus is defined by
$
u=0,
\,
v=0,
\,
\rho=0.
$
The real Jacobian of these three functions is
$$
DG=
\begin{pmatrix}
u_x&u_y&u_s&u_t\\
v_x&v_y&v_s&v_t\\
\rho_x&\rho_y&\rho_s&\rho_t
\end{pmatrix}.
$$
The critical curve is smooth at a point precisely when $\rank DG=3$, that is, when at least one of the four $3\times3$ minors of $DG$ is nonzero.
The logarithmic differential is
$$
D\Log=
\begin{pmatrix}
\dfrac{x}{x^2+y^2}&\dfrac{y}{x^2+y^2}&0&0\\
0&0&\dfrac{s}{s^2+t^2}&\dfrac{t}{s^2+t^2}
\end{pmatrix}.
$$
At a smooth point of the critical curve, the restricted map $h=\Log|_\Sigma$ fails to be immersive exactly when both determinants
$$
J_z=
\det
\begin{pmatrix}
u_x&u_y&u_s&u_t\\
v_x&v_y&v_s&v_t\\
\rho_x&\rho_y&\rho_s&\rho_t\\
\dfrac{x}{x^2+y^2}&\dfrac{y}{x^2+y^2}&0&0
\end{pmatrix}
$$
and
$$
J_w=
\det
\begin{pmatrix}
u_x&u_y&u_s&u_t\\
v_x&v_y&v_s&v_t\\
\rho_x&\rho_y&\rho_s&\rho_t\\
0&0&\dfrac{s}{s^2+t^2}&\dfrac{t}{s^2+t^2}
\end{pmatrix}
$$
vanish. Indeed, under $\rank DG=3$, the tangent line to $\Sigma$ is $\ker DG$. The differential of the restricted logarithmic map vanishes on this line if and only if both logarithmic coordinate differentials vanish there, which is equivalent to $J_z=J_w=0$.
A denominator-free form is obtained by defining
$$
\widetilde J_z=(x^2+y^2)J_z,
\qquad
\widetilde J_w=(s^2+t^2)J_w.
$$
These become real polynomial expressions after clearing the Laurent denominators in $f$, $A$, and $B$.


\medskip

Let
$
C=\{(z,w)\in(\C^\ast)^2\mid f(z,w)=0\}
$
be a smooth algebraic curve, where
$
f\in\C[z^{\pm1},w^{\pm1}].
$
Set
$
A=zf_z
$
and
$
B=wf_w.
$
Write
$
z=x+iy
$
and
$
w=s+it,
$
and define
$
u=\Rea f,
$
$
v=\Ima f,
$
and
$
\rho=\Ima(A\overline B).
$
On the chart
$
B\neq0,
$
the logarithmic critical locus is given by
$
u=v=\rho=0.
$
Let
$$
G=(u,v,\rho):\R^4\longrightarrow\R^3
$$
and denote its Jacobian matrix by
$$
DG=
\begin{pmatrix}
u_x&u_y&u_s&u_t\\
v_x&v_y&v_s&v_t\\
\rho_x&\rho_y&\rho_s&\rho_t
\end{pmatrix}.
$$
The logarithmic map is
$
\Log(z,w)=(\log|z|,\log|w|),
$
and its real differential is
$$
D\Log=
\begin{pmatrix}
\dfrac{x}{x^2+y^2}&\dfrac{y}{x^2+y^2}&0&0\\
0&0&\dfrac{s}{s^2+t^2}&\dfrac{t}{s^2+t^2}
\end{pmatrix}.
$$
Define
$$
J_z=
\det
\begin{pmatrix}
u_x&u_y&u_s&u_t\\
v_x&v_y&v_s&v_t\\
\rho_x&\rho_y&\rho_s&\rho_t\\
\dfrac{x}{x^2+y^2}&\dfrac{y}{x^2+y^2}&0&0
\end{pmatrix}
$$
and
$$
J_w=
\det
\begin{pmatrix}
u_x&u_y&u_s&u_t\\
v_x&v_y&v_s&v_t\\
\rho_x&\rho_y&\rho_s&\rho_t\\
0&0&\dfrac{s}{s^2+t^2}&\dfrac{t}{s^2+t^2}
\end{pmatrix}.
$$
Put
$
\widetilde J_z=(x^2+y^2)J_z
$
and
$
\widetilde J_w=(s^2+t^2)J_w.
$

\begin{proposition}
On the chart $B\neq0$, every degenerate critical lift belongs to the real algebraic set defined by
$
u=0,
\,
v=0,
\,
\rho=0
$
together with either the vanishing of all $3\times3$ minors of $DG$, or
$
\widetilde J_z=0,
\,
\widetilde J_w=0.
$
The analogous equations on the chart $A\neq0$ are obtained from the same critical equation and saturation by $|A|^2$ instead of $|B|^2$.
\end{proposition}

\begin{proof}
Let
$
q=(z,w)\in C
$
be a critical point of the logarithmic map on the chart
$
B(q)\neq0.
$
Since
$
C
$
is smooth, the logarithmic Gauss map is
$
\gamma_{\log}=[A:B].
$
On the chart
$
B\neq0,
$
the condition
$
\gamma_{\log}(q)\in\R\mathbb P^1
$
is equivalent to
$
A(q)/B(q)\in\R.
$
Because
$
B(q)\neq0,
$
this is equivalent to
$
A(q)\overline{B(q)}-\overline{A(q)}B(q)=0,
$
or equivalently
$
\rho(q)=\Ima(A(q)\overline{B(q)})=0.
$
Since
$
q\in C,
$
one also has
$
u(q)=v(q)=0.
$
Therefore every critical lift on this chart belongs to the real zero set
$
G^{-1}(0).
$

A degenerate critical lift is a critical point at which either the critical locus itself is singular or the restriction of the logarithmic map to the smooth critical locus fails to have maximal rank. In the plane-curve case the critical locus, whenever smooth, is one-dimensional over $\R$, so maximal rank for the restricted logarithmic map means rank one.

Assume first that the critical locus is singular at $q$ as a real algebraic set defined by
$
G=(u,v,\rho).
$
The Jacobian criterion gives
$
\rank DG(q)<3.
$
Since
$
DG(q)
$
is a $3\times4$ matrix, this is equivalent to the vanishing of all its $3\times3$ minors. Thus $q$ satisfies the first alternative in the proposition.

Assume now that the critical locus is smooth at $q$. Then
$
\rank DG(q)=3.
$
By the real implicit function theorem,
$
G^{-1}(0)
$
is a smooth real curve near $q$, and its tangent line is
$$
T_q\Sigma=\ker DG(q).
$$
Let
$
h=\Log|_\Sigma.
$
The point $q$ is degenerate for the restricted logarithmic map precisely when
$
\rank dh_q=0.
$
Because
$
T_q\Sigma
$
is one-dimensional, this means that both component differentials
$
d\log|z|_q
$
and
$
d\log|w|_q
$
vanish on
$
T_q\Sigma.
$

Denote by
$
\alpha_z
$
and
$
\alpha_w
$
the two row covectors of
$
D\Log(q),
$
so that
$$
\alpha_z=
\left(
\dfrac{x}{x^2+y^2},
\dfrac{y}{x^2+y^2},
0,
0
\right)
$$
and
$$
\alpha_w=
\left(
0,
0,
\dfrac{s}{s^2+t^2},
\dfrac{t}{s^2+t^2}
\right).
$$
Since
$
T_q\Sigma=\ker DG(q),
$
the condition
$
\alpha_z|_{T_q\Sigma}=0
$
is equivalent to saying that
$
\alpha_z
$
belongs to the row space of
$
DG(q).
$
Indeed, for any real matrix
$
M,
$
the annihilator of
$
\ker M
$
is exactly the row space of
$
M.
$
Because
$
\rank DG(q)=3,
$
the row space of
$
DG(q)
$
has dimension three in
$
(\R^4)^\ast.
$
Hence
$
\alpha_z
$
belongs to that row space if and only if the four rows consisting of the three rows of
$
DG(q)
$
and the row
$
\alpha_z
$
are linearly dependent. This is equivalent to
$
J_z(q)=0.
$
The same argument shows that
$
\alpha_w|_{T_q\Sigma}=0
$
if and only if
$
J_w(q)=0.
$
Therefore
$$
\rank dh_q=0
\quad\Longleftrightarrow\quad
J_z(q)=0
\ \text{and}\
J_w(q)=0.
$$

Since
$
q\in(\C^\ast)^2,
$
one has
$
x^2+y^2>0
$
and
$
s^2+t^2>0.
$
This means that 
$
J_z(q)=0 
\,\Longleftrightarrow\,
\widetilde J_z(q)=0
$
and
$
J_w(q)=0
\,\Longleftrightarrow\, 
\widetilde J_w(q)=0.
$
The multiplication by
$
x^2+y^2
$
and
$
s^2+t^2
$
removes the denominators coming from
$
D\Log.
$
After also multiplying by suitable monomials in
$
x,y,s,t
$
to clear the Laurent denominators in
$
f,
$
$
A,
$
and
$
B,
$
all defining equations become real polynomial equations. These multiplications do not change the zero set inside
$
(\C^\ast)^2,
$
because the factors used to clear denominators are nonzero there.
It follows that every degenerate critical lift on the chart
$
B\neq0
$
satisfies
$
u=v=\rho=0
$
and either all $3\times3$ minors of
$
DG
$
vanish, or
$
\widetilde J_z=\widetilde J_w=0.
$

It remains to justify the final chart statement. On the chart
$
A\neq0,
$
the projective reality condition
$
[A:B]\in\R\mathbb P^1
$
is still equivalent to
$
\Ima(A\overline B)=0.
$
Thus the same critical equation
$
\rho=0
$
is valid. The difference is only the chart condition. Algebraically, the chart
$
B\neq0
$
is enforced by saturation with respect to
$
|B|^2=B\overline B,
$
while the chart
$
A\neq0
$
is enforced by saturation with respect to
$
|A|^2=A\overline A.
$
The Jacobian and immersion-degeneracy arguments are unchanged. Therefore the analogous source-degeneracy system on
$
A\neq0
$
is obtained from the same equations with saturation by
$
|A|^2
$
instead of
$
|B|^2.
$
\end{proof}


The plane-curve multiple-lift system is also explicit. Introduce $q=(x,y,s,t)$ and $q'=(x',y',s',t')$. Let $u',v',\rho'$ denote the same expressions evaluated at $q'$. Then two distinct critical lifts with the same logarithmic image satisfy
$
u=v=\rho=0,
\,
u'=v'=\rho'=0,
$ \,
$
x^2+y^2=(x')^2+(y')^2,
\,
s^2+t^2=(s')^2+(t')^2,
$
and
$
(x-x')^2+(y-y')^2+(s-s')^2+(t-t')^2\neq0.
$
The inequality is implemented algebraically by saturation with respect to the last polynomial.

Introduce $R=x^2+y^2$ and $S=s^2+t^2$. Eliminating $x,y,s,t,x',y',s',t'$ from the saturated system produces an algebraic relation in $R,S$. Its positive real solutions give multiple-lift candidates in the contour through
$\di
(p_1,p_2)=\Big(\frac12\log R,\frac12\log S\Big).
$


\section{Plane-Curve Singular-Contour Computation Theorem}

Let $f\in\C[z^{\pm1},w^{\pm1}]$ and let
$C=\{(z,w)\in(\C^\ast)^2:f(z,w)=0\}$.  We assume that $C$ is smooth.  The restriction of the 
logarithmic map
 to $C$ will be denoted as before by $L_C$.  Put
$\Sigma=\Crit(L_C)$ and $h=\Log|_\Sigma$.  The contour of the amoeba of $C$
is $\CA_C=h(\Sigma)$.
The local properness hypothesis at $p\in\CA_C$ is understood in the following
germwise form.  There are neighborhoods $W$ of $p$ and $\Omega$ of
$h^{-1}(p)$ in $\Sigma$ such that $h|_\Omega:\Omega\to W$ is proper and,
after possibly shrinking $W$, every point of $\Sigma$ mapped into $W$ belongs
to $\Omega$.  Equivalently, the contour germ at $p$ is the image germ of this
proper restricted map.  This is the precise property used below to exclude
critical points escaping from the chosen source neighborhood while their
images converge to $p$.

Throughout the proof, a point $p\in\CA_C$ is called regular if there is an
open neighborhood $W$ of $p$ in $\R^2$ such that $W\cap\CA_C$ is a smooth
embedded real curve.  The singular locus $\Sing(\CA_C)$ is the complement of
the regular locus in $\CA_C$.  Thus the word singular refers to the image
germ of the contour and not merely to a singularity of an algebraic
elimination equation.

Write
$
z=x+iy,\, w=s+it,
$
and set
$
u=\Rea f,\, v=\Ima f,\,
A=zf_z,\, B=wf_w,\,
\rho=\Ima(A\overline B).
$
All these expressions are real algebraic after multiplication by a monomial
that clears the Laurent denominators.  Such multiplication does not alter
their zero sets in $(\C^\ast)^2$.

Since $C$ is smooth and $z,w$ are nonzero, $A$ and $B$ cannot vanish
simultaneously on $C$.  Indeed, $A=B=0$ would imply $f_z=f_w=0$, contrary to
the smoothness of $C$.  Hence the two logarithmic Gauss charts
$U_A=\{A\neq0\}$ and $U_B=\{B\neq0\}$ cover $C$.  On either chart the
critical-point condition is
$
u=v=\rho=0.
$
On $U_B$, for example, $\rho=0$ says that $A/B$ is real, which is equivalent
to $[A:B]\in\R\PP^1$.  This is precisely the logarithmic Gauss
characterization of the critical locus.

Let
$
G=(u,v,\rho):\R^4\longrightarrow\R^3
$
on the real coordinate domain corresponding to a logarithmic Gauss chart.
Its Jacobian is
$$
DG=
\begin{pmatrix}
u_x&u_y&u_s&u_t\\
v_x&v_y&v_s&v_t\\
\rho_x&\rho_y&\rho_s&\rho_t
\end{pmatrix}.
$$
The singular-Jacobian system consists of $u=v=\rho=0$ together with the
vanishing of all four $3\times3$ minors of $DG$.
The differential of the ambient logarithmic map is
$$
D\Log=
\begin{pmatrix}
\dfrac{x}{x^2+y^2}&\dfrac{y}{x^2+y^2}&0&0\\
0&0&\dfrac{s}{s^2+t^2}&\dfrac{t}{s^2+t^2}
\end{pmatrix}.
$$
Let $\alpha=d\log|z|$ and $\beta=d\log|w|$ denote its two rows.  Define
$\di
J_z=\det\begin{pmatrix}DG\\ \alpha\end{pmatrix},
\,
J_w=\det\begin{pmatrix}DG\\ \beta\end{pmatrix},
$
and define their denominator-free versions by
$\di
\widetilde J_z=(x^2+y^2)J_z,\, 
\widetilde J_w=(s^2+t^2)J_w.
$
After clearing the Laurent denominators already present in $u,v,\rho$, these
are real polynomials.  Since $z,w\neq0$, multiplying by the displayed
positive factors does not change their zero sets.

The two-lift system uses
$
q=(x,y,s,t),\, q'=(x',y',s',t').
$
Primes indicate evaluation at $q'$.  It consists of
$
u=v=\rho=0,\, u'=v'=\rho'=0,
$
together with the equal-radius equations
$
x^2+y^2=(x')^2+(y')^2,\,
s^2+t^2=(s')^2+(t')^2.
$
Put
$$
\delta(q,q')=(x-x')^2+(y-y')^2+(s-s')^2+(t-t')^2.
$$
Over $\R$, the inequality $\delta\neq0$ is exactly the condition $q\neq q'$.
The phrase \emph{exact real saturated two-lift system} will mean the
constructible real locus of the displayed equations on which $\delta\neq0$.
Equivalently, it is the real algebraic system obtained by adding a new real
variable $\lambda$ and the equation
$
\lambda\delta-1=0.
$
This formulation is equivalent to saturation for the purpose of exact
off-diagonal real projection and, unlike the real zero set of the saturated
ideal by itself, it does not reintroduce diagonal points belonging only to
the Zariski closure.

The chart conditions are understood in the same exact sense.  On $U_B$ one
may add a real variable $\mu$ and the equation
$
\mu|B|^2-1=0,
$
and on $U_A$ one uses $\mu|A|^2-1=0$.  Equivalently, one saturates by
$|B|^2$ or $|A|^2$ and retains the corresponding constructible open set.

Finally introduce radius variables
$
R=x^2+y^2,\, S=s^2+t^2.
$
The exact positive real projection means the set-theoretic semialgebraic
projection of the actual real solution sets to $(R,S)\in\R_{>0}^2$.  It is
not the real zero set of a complex or real Zariski elimination ideal.  Once
$R,S>0$ are known, their logarithmic image is
$$
\ell(R,S)=\left(\frac12\log R,\frac12\log S\right).
$$

\begin{lemma}
For $q\in C$, one has $q\in\Sigma$ if and only if
$[A(q):B(q)]\in\mathbb{RP}^1$.  On either of the charts $U_A$ and $U_B$, this is
equivalent to $u(q)=v(q)=\rho(q)=0$.
\end{lemma}

\begin{proof}
Let $q=(z,w)\in C$.  A tangent vector to $(\C^\ast)^2$ can be written in
logarithmic complex coordinates as $(\xi,\eta)\in\C^2$, meaning that the
corresponding infinitesimal variations are $\dot z=z\xi$ and $\dot w=w\eta$.
Differentiating $f$ gives
$
df_q(\dot z,\dot w)=A(q)\xi+B(q)\eta.
$
Consequently the complex tangent line $T_qC$ is the kernel of the nonzero
complex linear form $(\xi,\eta)\mapsto A\xi+B\eta$.  In these coordinates,
the differential of $\Log$ is
$
d\Log_q(\xi,\eta)=(\Rea\xi,\Rea\eta).
$

The real rank of $d(L_C)_q$ is smaller than two exactly when there is a
nonzero real covector $(a,b)\in\R^2$ that annihilates its image.  This means
that
$
a\Rea\xi+b\Rea\eta=0
$
for every $(\xi,\eta)$ satisfying $A\xi+B\eta=0$.  Equivalently, the real
linear functional $(\xi,\eta)\mapsto\Rea(a\xi+b\eta)$ vanishes on the complex
line $\ker(A,B)$.  This occurs exactly when $(a,b)$ is a nonzero real
multiple, in the projective sense, of $(A,B)$.  Hence
$[A:B]\in\mathbb{RP}^1$.
Since $A$ and $B$ do not vanish simultaneously, membership in
$\mathbb{RP}^1$ is equivalent to $A\overline B$ being real, or
$\Ima(A\overline B)=0$.  Together with the two real equations defining $C$,
this gives $u=v=\rho=0$.
\end{proof}

\begin{lemma}
Let $q\in\Sigma$ lie in one of the logarithmic Gauss charts.  If
$\rank DG(q)=3$, then $\Sigma$ is a smooth real analytic curve near $q$ and
$
T_q\Sigma=\ker DG(q).
$
This is equivalent to say that if $\Sigma$ is not a smooth curve germ at $q$, then every
$3\times3$ minor of $DG(q)$ vanishes.
\end{lemma}

\begin{proof}
On the chosen chart, $\Sigma$ is exactly $G^{-1}(0)$.  If
$\rank DG(q)=3$, the real analytic implicit-function theorem shows that
$G^{-1}(0)$ is a real analytic submanifold of codimension three in $\R^4$.
It therefore has dimension one, and its tangent line is $\ker DG(q)$.  The
last assertion is the contrapositive: if some $3\times3$ minor were nonzero,
then $DG$ would have rank three and the critical locus would be smooth.
\end{proof}

\begin{lemma}
Assume that $q\in\Sigma$ and $\rank DG(q)=3$.  Then
$
d(\Log|_\Sigma)_q=0
$
if and only if
$
\widetilde J_z(q)=\widetilde J_w(q)=0.
$
\end{lemma}

\begin{proof}
The tangent space $T_q\Sigma=\ker DG(q)$ is a one-dimensional real vector
space.  Choose a nonzero vector $\tau$ spanning it.  The restricted
logarithmic differential vanishes precisely when
$
\alpha(\tau)=0$ and $\beta(\tau)=0.$
Since $DG(q)$ has rank three, its row space has dimension three and its
annihilator is the line $\R\tau$.  For any row covector $\gamma$ on
$\R^4$, the determinant of the matrix obtained by adjoining $\gamma$ to
$DG(q)$ vanishes exactly when $\gamma$ lies in the row space of $DG(q)$.
The latter condition is equivalent to $\gamma(\tau)=0$.  Applying this with
$\gamma=\alpha$ and $\gamma=\beta$ gives
$$
J_z(q)=0\Longleftrightarrow\alpha(\tau)=0,
\qquad
J_w(q)=0\Longleftrightarrow\beta(\tau)=0.
$$
The factors $x^2+y^2$ and $s^2+t^2$ are strictly positive on
$(\C^\ast)^2$, so $J_z,J_w$ vanish exactly when
$\widetilde J_z,\widetilde J_w$ vanish.
\end{proof}

\begin{lemma}
Let $q\in\Sigma$ be a smooth point of $\Sigma$ such that
$d(\Log|_\Sigma)_q\neq0$.  Then there is a neighborhood $U$ of $q$ in
$\Sigma$ for which $\Log|_U$ is an embedding and $\Log(U)$ is a smooth
embedded real curve.
\end{lemma}

\begin{proof}
The domain $\Sigma$ is one-dimensional near $q$.  The nonzero differential
has rank one.  By the constant-rank theorem there are local coordinates
$\tau$ on $\Sigma$ and $(X,Y)$ on $\R^2$ in which the map is
$\tau\mapsto(\tau,0)$.  Shrinking the coordinate neighborhood makes the map
injective and a homeomorphism onto its image, with injective differential.
It is therefore an embedding onto a smooth embedded curve.
\end{proof}

\begin{lemma}
Let $p\in\CA_C$.  Assume that there are neighborhoods $\Omega\subset\Sigma$
of $h^{-1}(p)$ and $W\subset\R^2$ of $p$ such that
$h:\Omega\to W$ is proper and $h^{-1}(p)$ is finite.  Write
$
h^{-1}(p)=\{q_1,\ldots,q_m\}.
$
For every choice of pairwise disjoint neighborhoods
$U_j\subset\Omega$ of $q_j$, after shrinking them if necessary, there is a
neighborhood $W'\subset W$ of $p$ such that
$$
h^{-1}(W')\cap\Omega\subset\bigcup_{j=1}^mU_j.
$$
\end{lemma}

\begin{proof}
Suppose no such $W'$ existed.  Choose a decreasing neighborhood basis
$W_\nu$ of $p$ with compact closures contained in $W$.  There would be points
$$
r_\nu\in h^{-1}(W_\nu)\cap
\left(\Omega\setminus\bigcup_{j=1}^mU_j\right).
$$
Fix a compact neighborhood $K\subset W$ of $p$ containing all $W_\nu$ for
large $\nu$.  Properness makes $h^{-1}(K)\cap\Omega$ compact.  A subsequence
of $r_\nu$ converges to some $r_\infty\in\Omega$.  Since $h(r_\nu)\to p$,
continuity gives $h(r_\infty)=p$.  Thus $r_\infty=q_j$ for some $j$.  But
$U_j$ is a neighborhood of $q_j$, whereas every $r_\nu$ lies outside
$\bigcup_jU_j$, a contradiction.
\end{proof}

\begin{proposition}[Source of a singular contour point]
Under the hypotheses of the preceding lemma, if $p\in\Sing(\CA_C)$, then
either at least one $q\in h^{-1}(p)$ is a singular point of $\Sigma$ or a
critical point of $h$, or $h^{-1}(p)$ contains at least two distinct points.
\end{proposition}

\begin{proof}
It suffices to prove the contrapositive.  Suppose that $h^{-1}(p)=\{q\}$,
that $\Sigma$ is smooth at $q$, and that $dh_q\neq0$.  By the local embedding
lemma, there is a neighborhood $U$ of $q$ such that $h|_U$ is an embedding
onto a smooth embedded curve.  The isolation lemma gives a neighborhood
$W'$ of $p$ for which $h^{-1}(W')\cap\Omega\subset U$.  After shrinking $W'$
once more, the germwise local properness convention ensures that all contour
contributions near $p$ belong to $h(U)$.  Hence
$\CA_C\cap W'=h(U)\cap W'$ is a smooth embedded curve.  Thus $p$ is regular,
which proves the contrapositive.
\end{proof}

\begin{theorem}
Let $C=\{f(z,w)=0\}\subset(\C^\ast)^2$ be smooth. Assume local properness of
$\Log|_\Sigma$ near the contour point under consideration and finite critical
fibers. Every singular point of $\CA_C$ is obtained from one of the following
explicit real systems: the critical equations together with the
singular-Jacobian equations for $DG$; the critical equations together with
$\widetilde J_z=\widetilde J_w=0$; or the saturated two-lift system above.

If every pair of distinct nondegenerate critical lifts over the same
logarithmic point determines distinct local contour branches, and every
projected degenerate lift retained in the computation has a singular image
germ, then the exact positive real projection of the union of these systems,
followed by
$(R,S)\mapsto(\frac12\log R,\frac12\log S)$, is exactly the singular locus of
the contour.
\end{theorem}

\begin{proof}
Let $\mathscr E_{\mathrm{Jac}}$ be the exact real solution set, over both
logarithmic Gauss charts, of the equations
$$
u=v=\rho=0
$$
and all $3\times3$ minors of $DG$.  Let
$\mathscr E_{\mathrm{imm}}$ be the exact real solution set of
$$
u=v=\rho=\widetilde J_z=\widetilde J_w=0.
$$
The chart inequalities $|A|^2>0$ or $|B|^2>0$ and the torus inequalities
$x^2+y^2>0$, $s^2+t^2>0$ are retained.  Let
$\mathscr E_{\mathrm{off}}$ be the exact off-diagonal two-lift set.  Thus its
points are pairs $(q,q')$ satisfying both copies of the critical equations,
the two equal-radius equations, and $\delta(q,q')\neq0$.

Adjoin $R=x^2+y^2$ and $S=s^2+t^2$ to the first two systems.  For the
two-lift system use the common values
$$
R=x^2+y^2=(x')^2+(y')^2,\qquad
S=s^2+t^2=(s')^2+(t')^2.
$$
Let $E\subset\R_{>0}^2$ be the union of the exact radius projections of
$\mathscr E_{\mathrm{Jac}}$, $\mathscr E_{\mathrm{imm}}$, and
$\mathscr E_{\mathrm{off}}$.  We prove
$$
\Sing(\CA_C)=\ell(E).
$$

First consider $p\in\Sing(\CA_C)$.  By the source proposition, there are
three geometric possibilities.

Suppose that some lift $q\in h^{-1}(p)$ is a singular point of $\Sigma$.
Choose one of the charts $U_A,U_B$ containing $q$.  If $DG(q)$ had rank
three, the regularity lemma would make $\Sigma$ a smooth curve near $q$.
Therefore $\rank DG(q)<3$, so all four $3\times3$ minors of $DG(q)$ vanish.
The point $q$ belongs to $\mathscr E_{\mathrm{Jac}}$.

Suppose instead that $\Sigma$ is smooth at a lift $q$, but $q$ is a critical
point of $h$.  If $\rank DG(q)<3$, then $q$ already belongs to
$\mathscr E_{\mathrm{Jac}}$.  If $\rank DG(q)=3$, the determinantal immersion
criterion gives
$
\widetilde J_z(q)=\widetilde J_w(q)=0.
$
Hence $q\in\mathscr E_{\mathrm{imm}}$.

In the remaining possibility the fiber contains distinct points $q\neq q'$.
Both satisfy the critical equations.  Since $h(q)=h(q')=p$, one has
$
|z|=|z'|,\, |w|=|w'|,
$
and therefore the equal-radius equations hold.  Since $q\neq q'$, one has
$\delta(q,q')>0$.  Thus $(q,q')\in\mathscr E_{\mathrm{off}}$.  The two
logarithmic Gauss charts cover every lift, so possibly different chart choices
for $q$ and $q'$ cause no omission.
In all three cases, setting
$
R=|z|^2,\, S=|w|^2
$
gives a point $(R,S)\in E$, and
$$
p=(\log|z|,\log|w|)
=\left(\frac12\log R,\frac12\log S\right)=\ell(R,S).
$$
This proves
$
\Sing(\CA_C)\subset\ell(E).
$
For the reverse inclusion, let $(R,S)\in E$ and set $p=\ell(R,S)$.  Suppose
first that $(R,S)$ is represented by a degenerate lift in
$\mathscr E_{\mathrm{Jac}}\cup\mathscr E_{\mathrm{imm}}$.  The equality
$u=v=\rho=0$ ensures that this is an actual critical lift of $C$, not merely
a solution of an eliminated closure.  By the hypothesis of the theorem,
every projected degenerate lift retained in the computation has a singular
image germ.  Hence the contour germ at $p$ is singular, so
$p\in\Sing(\CA_C)$.

Suppose next that $(R,S)$ is represented by a point
$(q,q')\in\mathscr E_{\mathrm{off}}$.  The equation
$\lambda\delta-1=0$, or equivalently the retained inequality $\delta>0$,
guarantees that $q\neq q'$.  If either lift is degenerate, the common radius
point also occurs in one of the degenerate-lift systems, and the preceding
paragraph applies.  We may therefore assume that both lifts are
nondegenerate.

By the local embedding lemma, sufficiently small neighborhoods $U$ of $q$
and $U'$ of $q'$ map to smooth embedded contour branches
$
\Gamma=h(U),\,  \Gamma'=h(U')
$
through $p$.  By hypothesis these branches are distinct as germs.  If the
contour germ at $p$ were a smooth embedded curve germ $M$, then both
$\Gamma$ and $\Gamma'$ would be open one-dimensional submanifold germs of
$M$ through $p$.  In a sufficiently small connected coordinate interval of
$M$ containing $p$, each such branch contains a neighborhood of $p$ in $M$.
Consequently $\Gamma$ and $\Gamma'$ would coincide as germs, contradicting
the hypothesis.  Therefore the contour germ is not a smooth embedded curve,
and $p\in\Sing(\CA_C)$.
We have proved $\ell(E)\subset\Sing(\CA_C)$.  Together with the first
inclusion, this yields
$$
\Sing(\CA_C)=
\left\{
\left(\frac12\log R,\frac12\log S\right):
(R,S)\in E
\right\}.
$$
This is the asserted exact computation.
\end{proof}

\begin{proposition}
The equality in the theorem concerns the exact semialgebraic projection of
the real constructible systems.  Replacing it by the real zero set of a
Zariski elimination ideal gives, in general, only an algebraic candidate set
containing the required projection.
\end{proposition}

\begin{proof}
For an ideal $I\subset\R[X,R,S]$, the elimination ideal
$I\cap\R[R,S]$ defines the Zariski closure of the projection of the complex
algebraic set and need not define the exact projection of the real locus.
Projection can create inequalities, and a projected constructible set need
not be Zariski closed.  Likewise, the variety of
$I:\langle\delta\rangle^\infty$ is the Zariski closure of the part of
$V(I)$ lying off $\delta=0$; its closure can meet the diagonal.  The auxiliary
equation $\lambda\delta-1=0$ retains exactly $\delta\neq0$ before projection.
Real quantifier elimination, cylindrical algebraic decomposition, or exact
real-root isolation applied to the full system computes the required
semialgebraic projection.  The final restriction $R>0,S>0$ and the logarithm
are imposed only after this exact projection.
\end{proof}


\section{Singular Points of the Logarithmic Contour for\\
$z^2w^2+9w+z+3zw+1=0$}

For the curve $F(Z,W)=Z^2W^2+9W+Z+3ZW+1=0$, we explain precisely which equations produce the eleven corrected singular values of the logarithmic contour.  The calculation separates ramification of the restricted logarithmic map from coincidences of distinct critical lifts and from the exceptional conjugation-quotient cusp.  We give the real critical equations, the Gauss-parameter eliminant, the saturated two-lift system, the tangent and jet tests, rational isolating data, and a two-panel drawing of the non-flat amoeba and its complete contour.

 \subsection{Notation and background}

For a curve $C\subset(\C^*)^2$ we write $\Log(Z,W)=(\log|Z|,\log|W|)$ and call $\mathcal A_C=\Log(C)\subset\R^2$ its amoeba, a notion due to Gelfand, Kapranov and Zelevinsky \cite{GKZ}. If $C$ is the zero set of a Laurent polynomial $F$, its Newton polygon $\Delta_F$ is the convex hull of the exponents of the monomials of $F$. By Forsberg, Passare and Tsikh \cite{FPT} the complement $\R^2\setminus\mathcal A_C$ has finitely many components, all of them convex, and the order map, which attaches to each component a lattice point of $\Delta_F$, is injective; on a component the Ronkin function is affine with gradient equal to the order \cite{PR04}. A component whose order is an interior lattice point is bounded (a hole), a component whose order is a vertex is unbounded, and every vertex is the order of an unbounded component, because in the region where the monomial of a vertex dominates all the others no zero of $F$ can occur; see also the survey \cite{M04}. When the closure of $C$ in the toric surface of $\Delta_F$ is smooth, its genus is the number of interior lattice points of $\Delta_F$ \cite{Kho78}.

The logarithmic Gauss map of a smooth curve is $\gamma:C\to\mathbb{CP}^1$, $\gamma(Z,W)=[ZF_Z:WF_W]$, where $F_Z$ and $F_W$ are the partial derivatives; it extends to the compactification of $C$ with degree $2\Area(\Delta_F)$ \cite{Mikhalkin00}. The restriction $\Log|_C:C\to\R^2$ is a smooth map between surfaces, a point of $C$ is critical for it exactly when $\gamma$ is real there, and the contour of the amoeba, in the sense of Passare and Tsikh \cite{PT05}, is the set of critical values $\Log(\gamma^{-1}(\mathbb{RP}^1))$. The topological boundary $\partial\mathcal A_C$ is contained in the contour, but the contour may have pieces inside the amoeba. Since $F$ has real coefficients, the set $\gamma^{-1}(\mathbb{RP}^1)$ is invariant under complex conjugation and it contains the real part of $C$; this is the plane-hypersurface form of the logarithmic Gauss criterion developed by Mikhalkin and generalized in higher codimension by Madani and Nisse \cite{Mikhalkin00,MadaniNisse13}. We call real branch the set $R$ of points of $C$ with $Z$ and $W$ real, and complex branch the set $K$ of the points of $\gamma^{-1}(\mathbb{RP}^1)$ which are not real. Two conjugate points have the same image under $\Log$, so the complex branch contributes to the contour through its quotient by conjugation.

By Whitney \cite{Whitney55} the stable singularities of a smooth map between surfaces are folds and cusps. Here the critical set $\gamma^{-1}(\mathbb{RP}^1)$ is the fold curve, a cusp of the contour appears where the kernel of the differential is tangent to it, and, in addition to folds and cusps, there are two sources of singular points of the contour as a plane set: transversal self-crossings of the image, which we call nodes, and points where an arc of the image ends, which occur here at the points where the critical set $\gamma^{-1}(\mathbb{RP}^1)$ itself has a crossing, that is, at the real critical points of $\gamma$; we call them junctions and describe them in Section~5. We write $s=(s_1,s_2)=(\cos\psi,\sin\psi)$ and call $\psi\in\R/\pi\Z$ the slope of a critical point when $s_1ZF_Z=s_2WF_W$ there. Differentiating $F=0$ gives $ZF_Z\,d\log Z+WF_W\,d\log W=0$, hence along the critical set $d\log W=-\tan\psi\,d\log Z$; consequently the tangent to the contour at a point of slope $\psi$ is parallel to $(\cos\psi,-\sin\psi)$, and the contour has zero speed exactly where $\operatorname{Re}(d\log Z/d\psi)=0$ (when $\cos\psi\neq0$).

\subsection{The polynomial, its Newton polygon, the curve, and its amoeba}

The polynomial $F(Z,W)=Z^2W^2+9W+Z+3ZW+1$ has five monomials, with coefficients $1$ for $Z^2W^2$, $9$ for $W$, $1$ for $Z$, $3$ for $ZW$ and $1$ for $1$. The Newton polygon $\Delta$ is the quadrilateral with vertices $(0,0)$, $(1,0)$, $(2,2)$ and $(0,1)$, drawn in Figure~\ref{fig:newton}. Its sides are the vectors $(1,0)$, $(1,2)$, $(-2,-1)$ and $(0,-1)$, all primitive, so its boundary contains exactly $4$ lattice points; the shoelace formula gives $\Area(\Delta)=2$, and Pick's formula gives $2-4/2+1=1$ interior lattice point, namely $(1,1)$, the exponent of the monomial $3ZW$. The outer normals of the sides are $(0,-1)$, $(2,-1)$, $(-1,2)$ and $(-1,0)$, so the amoeba has four tentacles in these directions. The polynomials attached to the sides are $1+9W$, $Z+Z^2W^2$, $9W+Z^2W^2$ and $Z+1$ (up to the factor of the lowest monomial); each of them has a single simple root, so the closure of $C$ in the toric surface of $\Delta$ is transversal to the boundary at four points and avoids the torus-fixed points. The horizontal tentacle has the asymptote $\log|W|=-\log9$ (root $W=-1/9$ of $1+9W$) and the vertical one has the asymptote $\log|Z|=0$ (root $Z=-1$ of $Z+1$), as in the figures below.

\vspace{0.2cm}

Regard $F=Z^2W^2+(3Z+9)W+(Z+1)$ as a quadratic polynomial in $W$ over $\C(Z)$. Its discriminant is $D(Z)=(3Z+9)^2-4Z^2(Z+1)=-4Z^3+5Z^2+54Z+81$, with the three simple roots $4.873072$ and $-1.811536\pm0.934787\,i$. Above a value $Z_0\in\C^*$ with $D(Z_0)\neq0$ the curve has two distinct points, at which $F_W\neq0$; above a simple root of $D$ the double cover is simply branched, so $C$ is smooth there too. Hence $C$ is smooth in the torus. Since $D$ is not a square, $F$ is irreducible. Put $w=2Z^2W+3Z+9$; then $w^2=D(Z)$, so $C$ is birational to the elliptic curve $w^2=D(Z)$, which has genus $1$ because the cubic $D$ has simple roots and the double cover is also branched at $Z=\infty$. This agrees with the single interior lattice point of $\Delta$ \cite{Kho78}, and the cover is unramified above $Z=0$, where the two roots are $W\to\infty$ and $W=-1/9$. The real part of $C$ is contained in three of the four sign quadrants, since $F>0$ when $Z,W>0$. Since $D$ has only one real root, the real locus of the elliptic curve is connected; a genus one M-curve would have two real components, so $C$ is not an M-curve and therefore not a simple Harnack curve \cite{Mikhalkin00}. The real locus is a single oval through four points of the toric boundary, cut by the coordinate axes into arcs in the quadrants $(Z>0,W<0)$, $(Z<0,W<0)$ and $(Z<0,W>0)$; it has one point with vertical tangent in the $Z$-coordinate, the point $Z=4.873072$, $W=-0.497313$, whose image under $\Log$ is $(1.583725,-0.698536)$.

\subsection{The amoeba}

For fixed $u=\log|Z|$ the points of $C$ above it are given by the two roots $W_1,W_2$ of the quadratic in $W$ with $Z=e^{u+i\varphi}$, and the section of $\mathcal A_C$ is the union of the two intervals described by $\min(\log|W_1|,\log|W_2|)$ and $\max(\log|W_1|,\log|W_2|)$ when $\varphi$ runs over the circle; these two functions are continuous in $\varphi$, so each interval is exactly the image of a continuous function on a circle. We computed these sections for $1800$ values of $u$ in $[-9,9]$, using $6000$ values of $\varphi$ each, and also the analogous sections at fixed $v=\log|W|$ obtained from the quadratic in $Z$, which describe the horizontal tentacle accurately, and we took the union; every interval was given at least one pixel. The result is the amoeba shown in Figure~\ref{fig:newton} on the left in a speckled blue: the shade is a gradient from pale sky blue near the boundary to deep navy in the interior, blended with random pixel-scale noise, and it is decorative only. The four tentacles in the directions of the outer normals of $\Delta$ are visible, and so is one hole.

\begin{figure}[ht]
\centering
\includegraphics[width=0.32\textwidth]{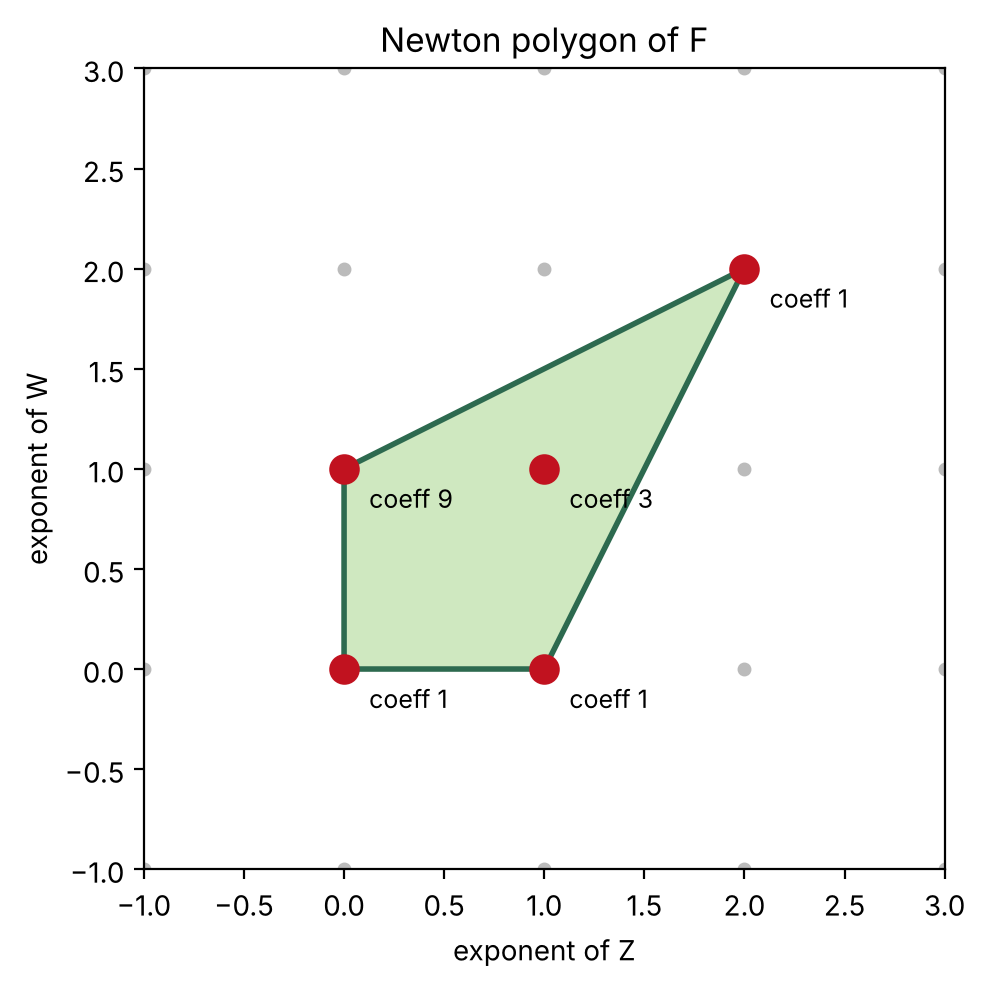}.  \qquad  \includegraphics[width=0.30\textwidth]{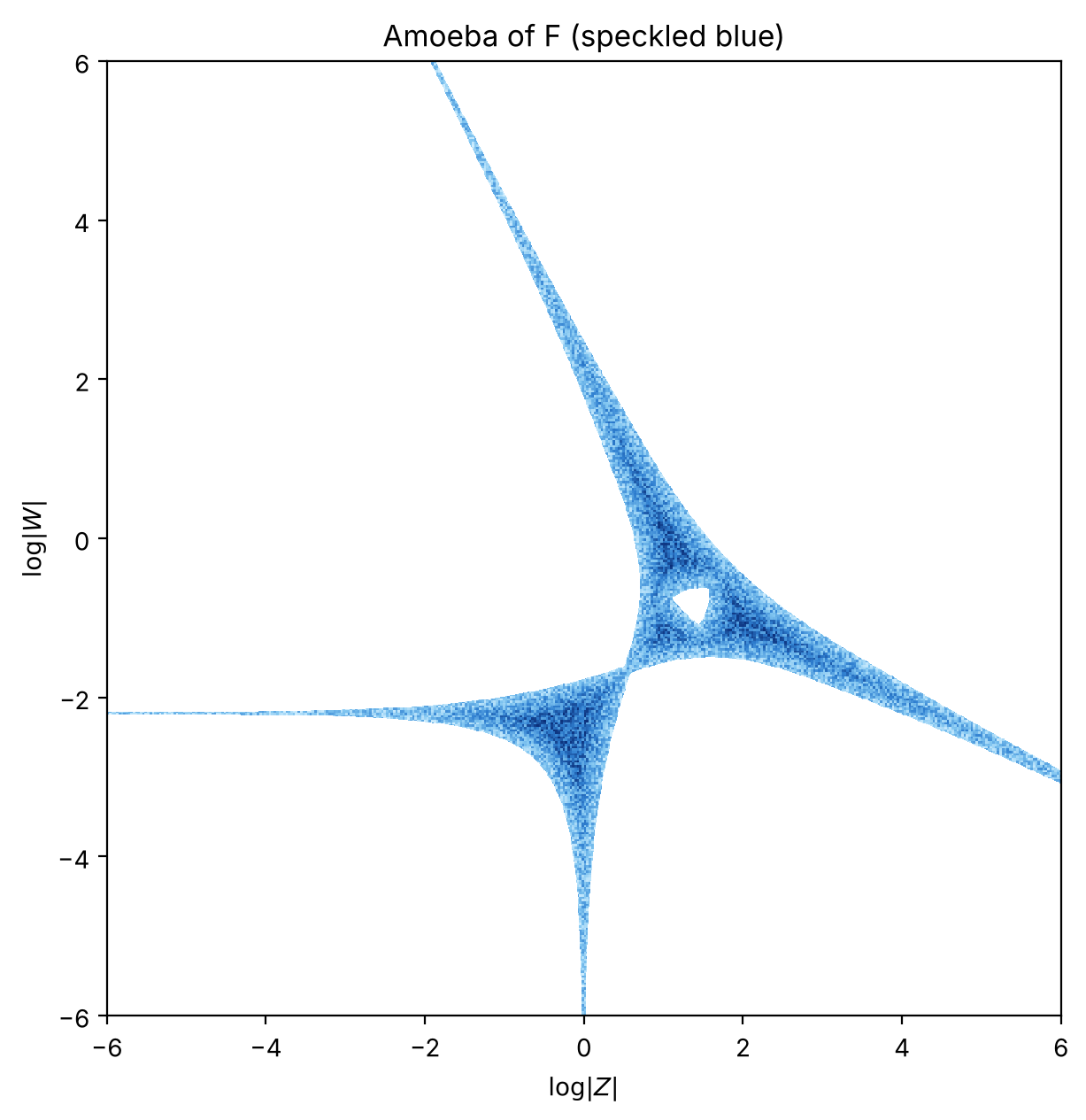}
\caption{On the left, the Newton polygon of $F$, with vertices $(0,0)$, $(1,0)$, $(2,2)$, $(0,1)$ and the interior lattice point $(1,1)$. The red dots are the exponents occurring in $F$ with their coefficients; the grey dots are the other lattice points. On the right, the amoeba of $F$ in the window $[-6,6]^2$ (speckled blue). The tentacles continue outside the window. }
\label{fig:newton}
\end{figure}

The complement has exactly five components: four unbounded ones, whose orders are the four vertices of $\Delta$, and one hole whose order is $(1,1)$. The count is at most five by \cite{FPT}, the four unbounded components exist as recalled in Section~1, and the hole is certified as follows. At the point $(u,v)=(1.40,-0.85)$ the moduli of the five terms of $F$ are $3.004$, $3.847$, $4.055$, $5.200$ and $1$; the minimum of $|F|$ on a grid of $3000\times3000$ points of the fibre torus is $0.9871$, the Lipschitz constants of $F$ with respect to the two arguments are bounded by $15.27$ and $15.06$, and the grid spacing is $0.00209$, so $|F|\geq0.955>0$ on the whole fibre torus and the point lies in the complement of the amoeba. It does not belong to an unbounded component by the labelling of the picture, so it belongs to the fifth component, which is the hole. Labelling the connected components of the complement of the computed set in the window $[-9,9]^2$ gives five components, of which four reach the border and one is bounded. The area of $\mathcal A_C$, computed as the integral over $u\in[-30,30]$ of the length of the section, is $8.288$, the same up to $10^{-3}$ for $[-20,20]$, whereas the maximal area for this polygon is $\pi^2\Area(\Delta)=2\pi^2\approx19.739$ \cite{PR04}; the amoeba fills about $42\%$ of it, which confirms that $C$ is not a Harnack curve \cite{MR01}.

\subsection{The contour: computation and the real and complex branches}

On $C$ we may replace $Z^2W^2$ by $-(9W+Z+3ZW+1)$ in the logarithmic derivatives and obtain $ZF_Z=-18W-Z-3ZW-2$ and $WF_W=-9W-2Z-3ZW-2$. The condition $s_1ZF_Z=s_2WF_W$ becomes bilinear in $(Z,W)$, namely $W=N_\psi(Z)/M_\psi(Z)$ with $N_\psi(Z)=(s_1-2s_2)Z+2(s_1-s_2)$ and $M_\psi(Z)=3(s_2-s_1)Z+9(s_2-2s_1)$. Substituting into $F=0$ and clearing denominators, the critical points of slope $\psi$ are the points $(Z,N_\psi(Z)/M_\psi(Z))$ where $Z$ is a root of the quartic polynomial $P_\psi(Z)=Z^2N_\psi^2+(3Z+9)N_\psi M_\psi+(Z+1)M_\psi^2$; its degree $4=2\Area(\Delta)$ is the degree of $\gamma$ \cite{Mikhalkin00}. For $\psi=3\pi/4$, for instance, one finds up to a factor $P(Z)=9Z^4+6Z^3-101Z^2-216Z-243$.

We solved $P_\psi=0$ numerically for $60000$ slopes in $(0,\pi)$ (and for $200001$ slopes to count real roots). For every slope there are four roots; either exactly two of them are real and the other two are complex conjugate, or all four are real. The number of real roots changes at exactly four slopes, $\psi_1=0.358783$, $\psi_2=0.526911$, $\psi_3=1.043885$ and $\psi_4=1.212013$: there are four real roots for $\psi\in(\psi_1,\psi_2)\cup(\psi_3,\psi_4)$ and two for the other slopes. At each of these slopes two real roots collide, which is a real critical point of $\gamma$. All the critical points of $\gamma$ can be located algebraically: put $t=\tan\psi$; the discriminant of $P_\psi$ with respect to $Z$ is, up to a constant, $(4t^2-11t+4)^2\,Q(t)$ with
$Q(t)=144t^8+184t^7-1853t^6-1624t^5+7594t^4-1624t^3-1853t^2+184t+144$.
We computed it exactly by evaluating the resultant of $P_\psi$ and $\partial_ZP_\psi$ in rational arithmetic at thirteen integer values of $t$ and interpolating. The factor $4t^2-11t+4$ is spurious: its roots $t=0.43127$ and $t=2.31873$ are the slopes where $N_\psi$ and $M_\psi$ have a common root, so that the formula $W=N_\psi/M_\psi$ is indeterminate. The polynomial $Q$ is palindromic and equals $t^4q(t+1/t)$ with $q(x)=144x^4+184x^3-2429x^2-2176x+11588$; the roots of $q$ are $x=2.300652$, $x=3.041580$ and $x=-3.310005\pm0.737441\,i$. The two real values of $x$ give the four real slopes $t=\tan\psi_i$ with $\tan\psi_1=0.375014$, $\tan\psi_2=0.581775$, $\tan\psi_3=1.718876$ and $\tan\psi_4=2.666565$, and the complex values give four non-real slopes $t=-2.99997\pm0.82244\,i$ and $t=-0.31004\pm0.08500\,i$. The Riemann--Hurwitz formula for $\gamma$ of degree $4$ on a curve of genus $1$ gives $0=4\cdot(-2)+r$, hence a total ramification $r=8$, in agreement with these eight simple ramification slopes, four of them real; the other four critical points of $\gamma$ have non-real value and are not on the critical set.

For slopes with exactly two real roots the other two roots are complex conjugate and give one contour point (their common image under $\Log$) per slope. As $\psi$ increases from $\psi_4$ to $\pi$ and then, since $\psi=\pi$ and $\psi=0$ are the same slope, from $0$ to $\psi_1$, this point traces one continuous arc, which we call the arc $A$, from the junction point $J_4$ to the junction point $J_1$ (the two ends of the slope interval $(\psi_4,\pi)$ and $(0,\psi_1)$ match at $\psi=0\equiv\pi$ at the point $(1.591721,-1.485009)$, where the tangent of the contour is horizontal). For $\psi\in(\psi_2,\psi_3)$ it traces a second arc $B$ from $J_2$ to $J_3$. Together with the conjugate arcs, $A$ and $B$ lift to two closed curves in $C$ which cross the real branch at $J_1,J_4$ and at $J_2,J_3$. The images of the real roots, taken in the three sign quadrants, form the image of the real branch.


\begin{figure}[ht]
\centering
\includegraphics[width=0.32\textwidth]{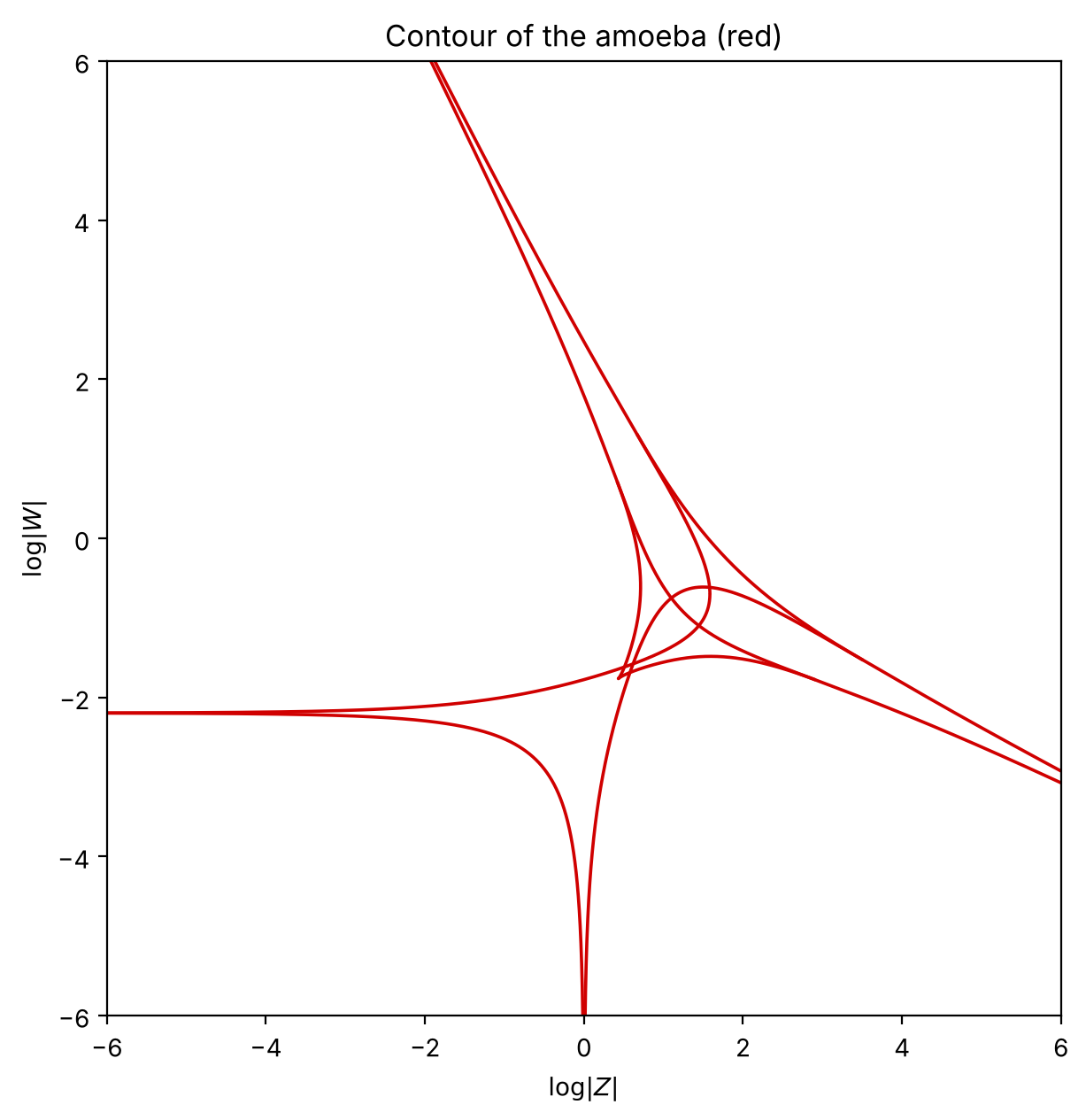} \qquad \includegraphics[width=0.32\textwidth]{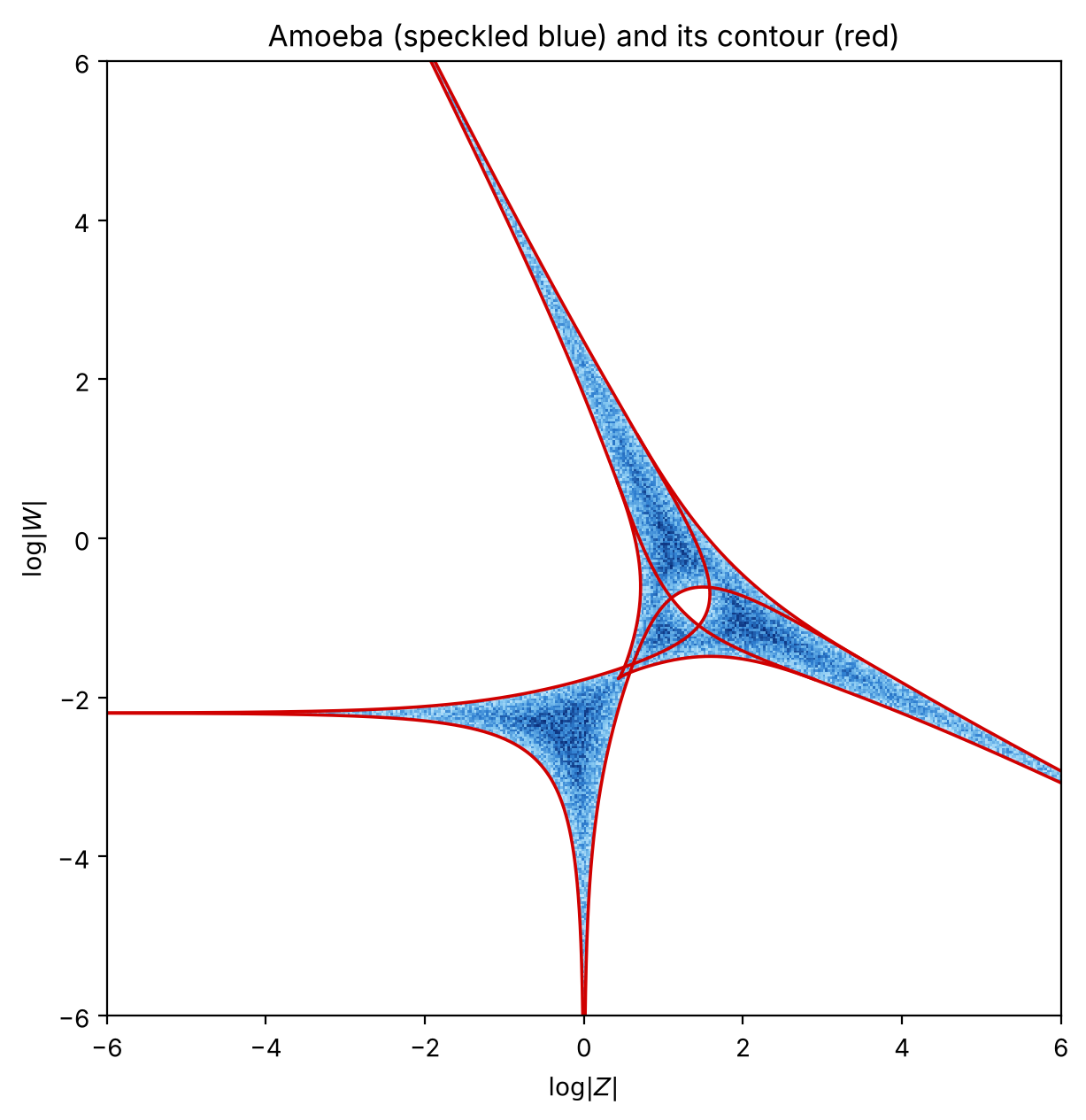}
\caption{On the left, the contour of the amoeba of $F$, in red, in the window $[-6,6]^2$: the images under $\Log$ of the real branch and of the complex branch. The picture on the right, the amoeba (speckled blue) and its contour (red) in the same figure. }
\label{fig:contour}
\end{figure}

\subsection{The singular points of the contour}

Table~\ref{tab:sing} lists the coordinates of all eleven singular points, and the paragraphs below explain how each was obtained.

\paragraph{Junctions.} At each of the four slopes $\psi_i$ a real critical point of $\gamma$ occurs; call it $J_i$. Near it we may write $\gamma-\psi_i=c\,z^2$ in a local coordinate $z$ that is real on the real branch, so the real branch is $z\in\R$ and the complex branch is $z\in i\R$. Since $F$ is real, the Taylor coefficients of $h=(\log Z,\log W)$ at $z=0$ are real, $h(z)=h_0+c_1z+c_2z^2+\cdots$, and the image of the real branch is $\Log h_0+c_1t+c_2t^2+\cdots$ while the image of the complex branch is $\Log h_0-c_2t^2+\cdots$ with the first-order term absent. The derivative $h'(z)$ is a multiple of $(1,-\tan\psi(z))$ with $\psi(z)=\psi_i+O(z^2)$, so $c_2$ is parallel to $c_1$. Hence the image of the real branch has an inflection point at $\Log h_0$ (it is a point where the contour of the real branch has zero curvature), and the image of the complex branch is a half-line-like arc starting at $\Log h_0$ in the direction of the tangent of the real branch and ending there: the conjugate points $z=\pm it$ give the same image point, so the arc is traversed twice. The contour near $J_i$ is therefore a smooth arc with a second arc attached tangentially at one point, which is not a one-dimensional manifold, so $J_i$ is a singular point of the contour as a set. Numerically the angle between the two tangents is smaller than $0.01^\circ$. The four junctions are $J_1$, with real point $(Z,W)=(-18.1126,-0.169515)$ and image $(2.896608,-1.774811)$; $J_2$, with $(-32.9168,0.217972)$ and image $(3.493983,-1.523387)$; $J_3$, with $(1.961751,-3.657421)$ and image $(0.673837,1.296758)$; and $J_4$, with $(-1.525639,-2.012511)$ and image $(0.422413,0.699383)$. They lie in the quadrants $(-,-)$, $(-,+)$, $(+,-)$ and $(-,-)$ respectively. The arc $A$ joins $J_4$ to $J_1$ and the arc $B$ joins $J_2$ to $J_3$.

\paragraph{The cusp.} Along the complex branch the derivative of $\Log$ with respect to $\psi$ vanishes only at $\psi=3\pi/4$. Indeed, on each of the arcs we scanned the speed of the contour with spacing $2.5\cdot10^{-5}$ in $\psi$; the only slope with speed below $0.5$ is $\psi=3\pi/4$ on the arc $A$, where the speed drops to $10^{-4}$ (the sign change of $d\log|Z|/d\psi$ at $\psi=\pi/2$ is not a cusp: there $\cos\psi=0$, the tangent is vertical and the speed is $2.3$). At $\psi=3\pi/4$ the quartic $9Z^4+6Z^3-101Z^2-216Z-243$ has the real roots $4.022061$ and $-2.827869$ and the complex conjugate roots $-0.930429\pm1.228074\,i$; the latter give the cusp point $Z=-0.9304293+1.2280737\,i$, $W=(-3Z-4)/(6Z+27)=-0.1033810-0.1364526\,i$ and $\Log=(0.4322594,-1.7649652)$. Computing $d\log Z/d\psi$ in extended precision gives $-1.01602\,i$ up to an error of order $10^{-10}$ in the real part, so the real part vanishes and the cusp is located at $\psi=3\pi/4$ exactly as far as our precision goes. A fit of the arc gives, with $\delta=\psi-3\pi/4$, $u-0.4322594=0.3565\,\delta^2+0.2377\,\delta^3+1.118\,\delta^4+\cdots$ and $v+1.7649652=0.3565\,\delta^2-0.2377\,\delta^3+1.118\,\delta^4-\cdots$: the quadratic terms are along $(1,1)$, in agreement with the tangent direction $(\cos\psi,-\sin\psi)\parallel(1,1)$, and the cubic terms along $(1,-1)$ are non-zero, so this is an ordinary (semicubical) cusp pointing to the lower left. It comes from the complex branch, since the roots involved are not real. It lies in the interior of the amoeba, at distance $0.116$ from the complement.

\paragraph{The nodes.} We looked for transversal intersections of the images of the pieces of the contour in two ways. First, every real piece was sampled by the exact parametrization $Z=\pm e^t$, $W$ a real root of the quadratic (using cancellation-free formulas for the root of small modulus, which matter along the horizontal tentacle where the two real images differ by less than $10^{-5}$), and the complex arcs were sampled in $\psi$; all pieces were resampled with spacing $0.004$, the pairs of points within distance $0.008$ were located with a $k$-d tree, and proper crossings of neighbouring segments were retained, in the window $|u|,|v|\le14$ (and with spacing $0.0025$ in $|u|,|v|\le9$). Second, each crossing was refined with Newton's method on the exact equations. There are exactly six crossings, and none of them involves two pieces of the complex branch. Far outside the window the images of the tentacles are asymptotic to straight lines and are ordered, so no further crossing occurs there.

The node $N_1=(0.6174005,-1.5798241)=(\log\tfrac{3}{\varphi},-\log3\varphi)$, where $\varphi=(1+\sqrt5)/2$, is the crossing of the images of the real points $(3/\varphi,\,-1/(3\varphi))$ in the quadrant $(+,-)$ (slope $\psi=2.787874$) and $(-3/\varphi,\,1/(3\varphi))$ in $(-,+)$ (slope $1.924515$), at an angle of $49.467^\circ$. The node $N_2=(\log3,\log(\sqrt2/3))=(1.0986123,-0.7520387)$ is the crossing of the images of $(-3,\sqrt2/3)$ in $(-,+)$ (slope $2.460385$) and $(-3,-\sqrt2/3)$ in $(-,-)$ (slope $0.919043$): indeed $F(-3,W)=9W^2-2$. Its angle is $88.312^\circ$. The node $N_3=(\log3\sqrt2,-\log3)=(1.4451859,-1.0986123)$ is the crossing of the images of $(3\sqrt2,-1/3)$ in $(+,-)$ (slope $2.252004$) and $(-3\sqrt2,-1/3)$ in $(-,-)$ (slope $0.651753$), again at an angle of $88.312^\circ$. The node $N_4=(\log3\varphi,-\log\tfrac3\varphi)=(1.5798241,-0.6174005)$ is the crossing of the images of $(3\varphi,-\varphi/3)$ in $(+,-)$ (slope $1.477680$) and $(-3\varphi,\varphi/3)$ in $(-,+)$ (slope $0.093116$), at an angle of $79.330^\circ$. These four nodes come from two different arcs of the real branch, so they come from the real branch alone. The two remaining nodes involve the arc $A$ of the complex branch. The node $N_5=(0.5089466,-1.6187491)$ is the crossing of the image of the real point $(1.663538,-0.198146)$ in $(+,-)$ (slope $2.805699$) with the image of the complex point $(-1.422714+0.862115\,i,\,-0.010381-0.197874\,i)$ of slope $1.955100$, at an angle of $48.736^\circ$; the node $N_6=(0.5784755,-1.6882780)$ is the crossing of the image of the real point $(-1.783318,0.184838)$ in $(-,+)$ (slope $1.906690$) with the image of the complex point $(-0.093428+1.780869\,i,\,-0.158079-0.095791\,i)$ of slope $2.757289$, again at $48.736^\circ$. In each node the two branches have different slopes, hence different tangents, which confirms that the crossings are transversal.

\begin{figure}[ht]
\centering
\includegraphics[width=0.32\textwidth]{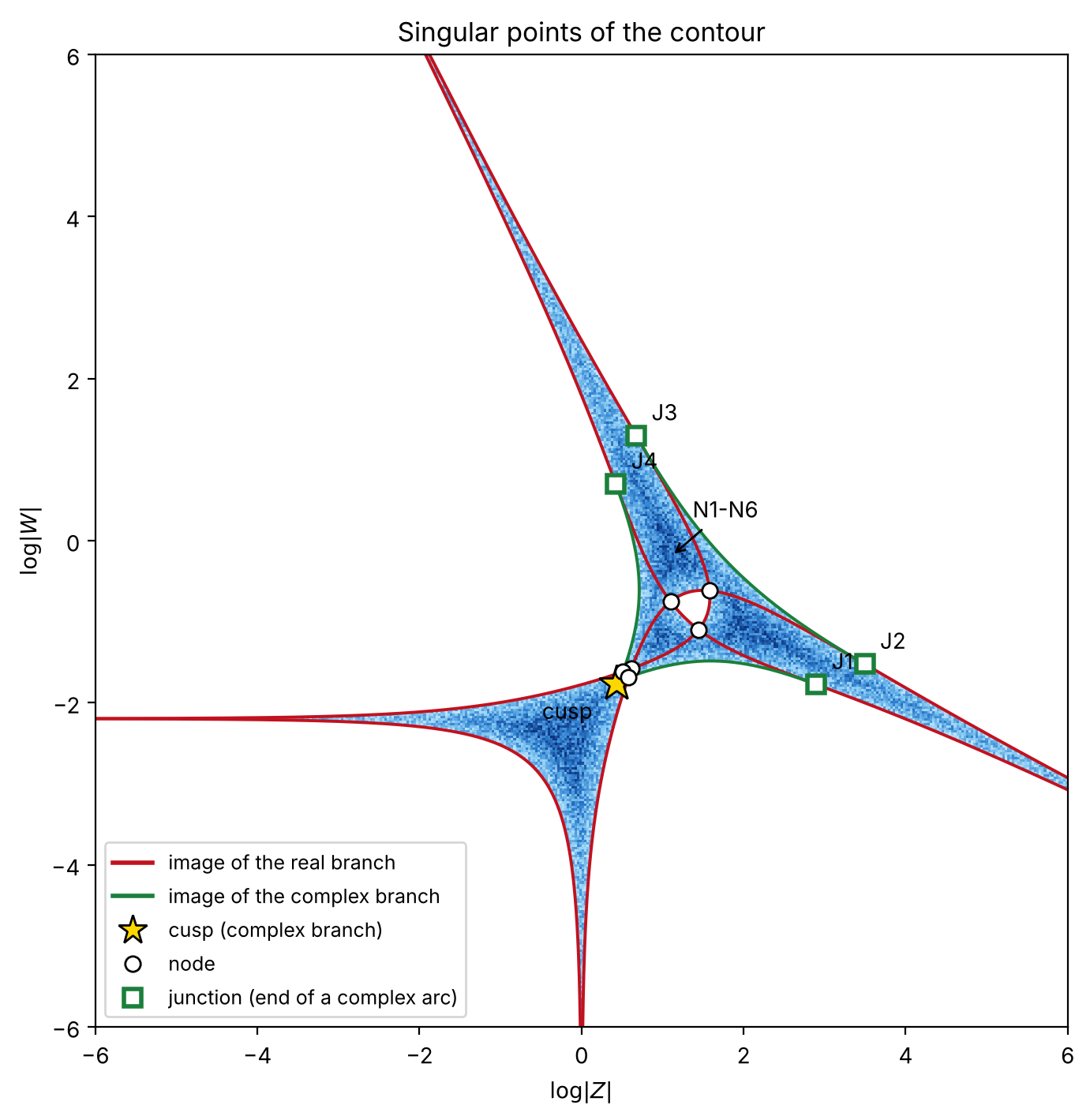} \qquad  \includegraphics[width=0.32\textwidth]{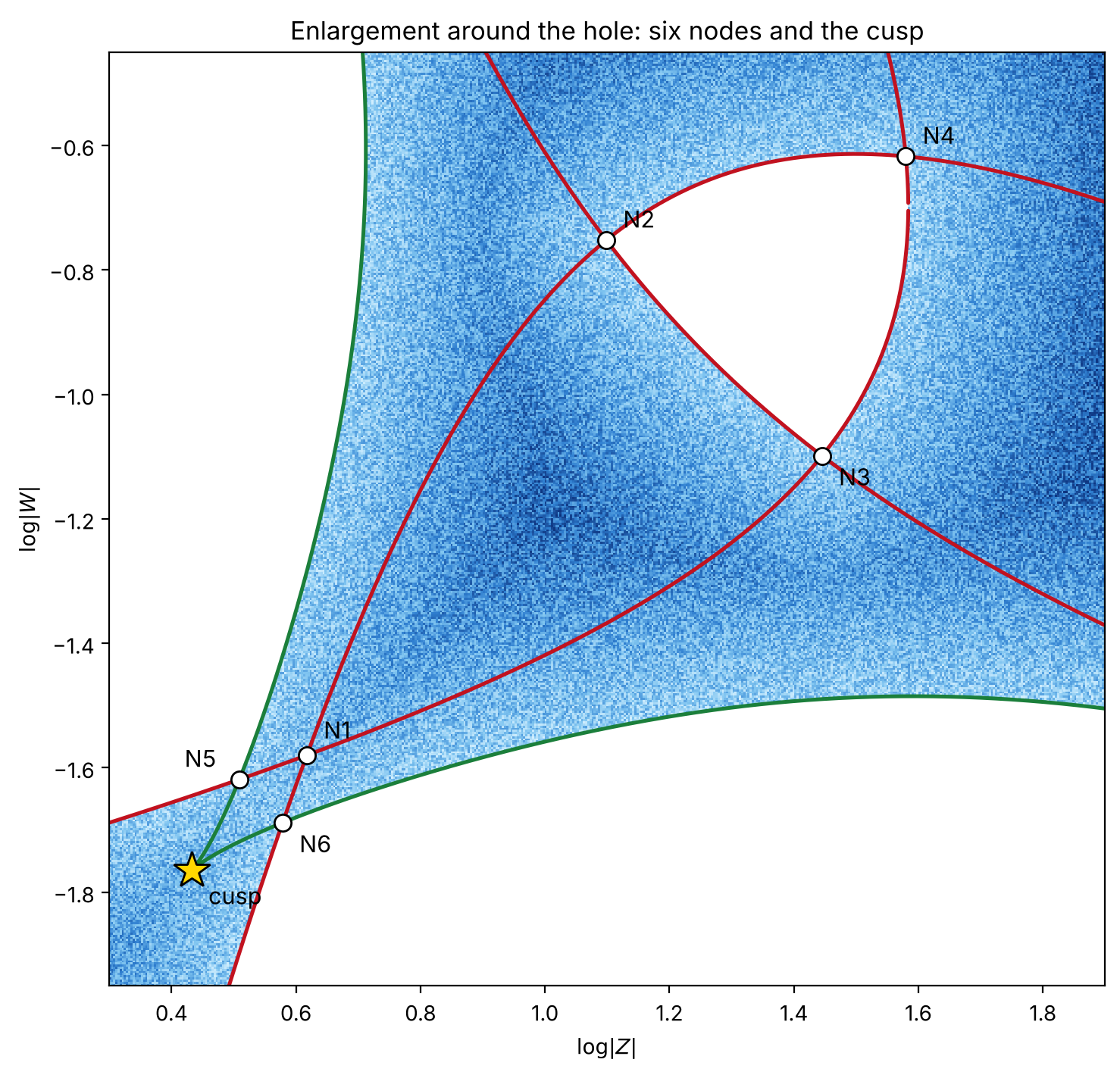}
\caption{On the left, the singular points of the contour. The image of the real branch is drawn in crimson, the image of the complex branch in green; the star is the cusp, the white circles are the nodes $N_1,\dots,N_6$ (clustered near the hole) and the white squares are the junctions $J_1,\dots,J_4$. On the right, enlargement around the hole. The hole is the curvilinear triangle with corners $N_2$, $N_3$, $N_4$, whose sides are arcs of the real branch in the quadrants $(-,+)$ (from $N_2$ to $N_4$), $(+,-)$ (from $N_4$ to $N_3$, through the point with $Z=4.873072$) and $(-,-)$ (from $N_3$ to $N_2$). The cusp of the green arc is at the lower left.}
\label{fig:sing}
\end{figure}

\begin{table}[ht]
\centering
\footnotesize
\setlength{\tabcolsep}{4pt}
\renewcommand{\arraystretch}{1.35}
\begin{tabular}{@{}lll>{\raggedright\arraybackslash}p{4.9cm}@{}} 
\hline
Name & Type & $(u,v)=\Log(Z,W)$ & Point(s) $(Z,W)$ of $C$ above it \\
\hline
$J_1$ & junction (real, arc $A$ ends) & $(2.896608,\,-1.774811)$ & $(-18.1126,\,-0.169515)$ \\
$J_2$ & junction (real, arc $B$ ends) & $(3.493983,\,-1.523387)$ & $(-32.9168,\,0.217972)$ \\
$J_3$ & junction (real, arc $B$ ends) & $(0.673837,\,1.296758)$ & $(1.961751,\,-3.657421)$ \\
$J_4$ & junction (real, arc $A$ ends) & $(0.422413,\,0.699383)$ & $(-1.525639,\,-2.012511)$ \\
$N_1$ & node (real--real) & $(0.6174005,\,-1.5798241)$ & $(3/\varphi,\,-1/(3\varphi))$ \newline $(-3/\varphi,\,1/(3\varphi))$ \\
$N_2$ & node (real--real) & $(1.0986123,\,-0.7520387)$ & $(-3,\,\sqrt2/3)$ \newline $(-3,\,-\sqrt2/3)$ \\
$N_3$ & node (real--real) & $(1.4451859,\,-1.0986123)$ & $(3\sqrt2,\,-1/3)$ \newline $(-3\sqrt2,\,-1/3)$ \\
$N_4$ & node (real--real) & $(1.5798241,\,-0.6174005)$ & $(3\varphi,\,-\varphi/3)$ \newline $(-3\varphi,\,\varphi/3)$ \\
$N_5$ & node (real--complex, arc $A$) & $(0.5089466,\,-1.6187491)$ & $(1.663538,\,-0.198146)$ \newline $Z=-1.422714+0.862115\,i$, \newline $W=-0.010381-0.197874\,i$ \\
$N_6$ & node (real--complex, arc $A$) & $(0.5784755,\,-1.6882780)$ & $(-1.783318,\,0.184838)$ \newline $Z=-0.093428+1.780869\,i$, \newline $W=-0.158079-0.095791\,i$ \\
$K$ & cusp (complex, arc $A$) & $(0.4322594,\,-1.7649652)$ & $Z=-0.9304293+1.2280737\,i$, \newline $W=-0.1033810-0.1364526\,i$ \\
\hline
\end{tabular}
\caption{Coordinates of all the singular points of the contour. The columns give the image in the plane $(u,v)=(\log|Z|,\log|W|)$ and the point or points of the curve that are mapped to it (the complex conjugate of a non-real point gives the same image). Here $\varphi=(1+\sqrt5)/2$. The preimages of $N_1,\dots,N_4$ are exact, and so are their images, $N_1=(\log\tfrac3\varphi,-\log3\varphi)$, $N_2=(\log3,\log\tfrac{\sqrt2}{3})$, $N_3=(\log3\sqrt2,-\log3)$ and $N_4=(\log3\varphi,-\log\tfrac3\varphi)$; the other values are rounded to the digits shown.}
\label{tab:sing}
\end{table}

\noindent {\it Position with respect to the amoeba.} We measured the distance of each singular point to the complement of the amoeba on a mask with pixel size $0.0008$. The nodes $N_2$, $N_3$, $N_4$ (the corners of the hole) and $N_5$, $N_6$ are within $0.006$ of the complement, that is, they are on the boundary of the amoeba at the resolution of the picture. The node $N_1$ lies in the interior, at distance $0.089$, and so does the cusp, at distance $0.116$. The arms of the cusp become boundary arcs after they cross the real arcs at $N_5$ and $N_6$.

\subsection{Summary}

In the finite plane the contour has $11$ singular points. There is one cusp, of ordinary type, and it comes from the complex branch. There are six nodes: four ($N_1,N_2,N_3,N_4$) are crossings of two arcs of the real branch lying in different sign quadrants or, for $N_2$, in the quadrants $(-,+)$ and $(-,-)$ above the same value $Z=-3$, and two ($N_5,N_6$) are crossings of the real branch with the complex arc $A$; no node comes from two arcs of the complex branch. There are four junctions $J_1,\dots,J_4$, the real critical points of the logarithmic Gauss map, where an arc of the complex branch ends tangentially on the real branch at a point where the image of the real branch has an inflection. If one counts only the singularities of a plane curve in the usual sense, that is, cusps and nodes, the answer is seven: one cusp and six nodes. At infinity the contour has eight ends along the four tentacles.

The statements proved from the cited theorems and exact computations are the description of the Newton polygon, the smoothness, irreducibility and genus of $C$, the fact that $C$ is not an M-curve, the equation $P_\psi$ of the critical points, the exact polynomial $Q$, the local analysis at the junctions, and the exact coordinates of $N_1,\dots,N_4$ (we checked that $F$ vanishes there). The existence of the hole is certified by an explicit bound. The location and number of the cusp and of the nodes $N_5,N_6$, the absence of other crossings and of other cusps, and the area of the amoeba are numerical.



\section{Maximal Sparsity Does Not Maximize the Number of Contour Singularities}
 
 In general, one cannot prove that a polynomial supported exactly on the
vertices of its Newton polygon has the largest possible number of
singularities of the amoeba contour.  The statement is false, even for a
lattice triangle.
Let
$
\Delta=\operatorname{conv}\{(0,0),(3,0),(0,1)\}.
$
Every polynomial supported exactly on the vertices of $\Delta$ has the
form $g(z,w)=a+bz^3+cw$, where $a,b,c\in\C^\ast$.  Its logarithmic
contour has no finite singular point.  On the other hand, the polynomial
$f(z,w)=w-1+z-z^3$ has the same Newton polygon $\Delta$, but its contour
has an ordinary transverse node at $(0,0)$.  Consequently,
$$
\sup_{\supp(g)=\operatorname{Vert}(\Delta)}
\#\Sing_{\mathrm{isol}}\calC(\calA_g)
=0
<
1
\leq
\sup_{\Newt(f)=\Delta}
\#\Sing_{\mathrm{isol}}\calC(\calA_f).
$$
Thus maximally sparse polynomials do not, in general, maximize the number
of contour singularities.

\subsection*{The contour of a line}

Consider the line
$L=\{1+U+V=0\}\subset(\C^\ast)^2$.  Its logarithmic Gauss map is
$\gamma_L(U,V)=[U:V]$.  A point of $L$ is critical for
$\Log|_L$ exactly when $[U:V]\in\mathbb P^1_\R$.  Since
$V=-1-U$, this condition is equivalent to
$U/(-1-U)\in\R$.
Write $U=x+iy$.  Direct calculation gives
$\di
\operatorname{Im}\left(\frac{U}{-1-U}\right)
=\frac{-y}{|1+U|^2}.
$
Hence the criticality condition is $y=0$.  The critical locus is
therefore parametrized by
$t\in\R\setminus\{-1,0\}$ through
$(U,V)=(t,-1-t)$.  Its logarithmic image is parametrized by
$$
\rho(t)=\bigl(\log|t|,\log|1+t|\bigr).
$$

This parametrization is immersive because
$\rho'(t)=(1/t,1/(1+t))$, and the two components of this vector cannot
vanish simultaneously.  It is also injective.  Indeed, if
$\rho(t)=\rho(s)$, then $|t|=|s|$, so either $s=t$ or $s=-t$.  In the
second case the equality $|1+t|=|1-t|$ implies $t=0$, which is excluded.
Therefore $s=t$.  It follows that the three intervals
$(-\infty,-1)$, $(-1,0)$ and $(0,\infty)$ map to three embedded,
pairwise disjoint real-analytic contour branches.  Therefore,
$
\Sing_{\mathrm{isol}}\calC(\calA_{1+U+V})=\varnothing.
$

\subsection*{Vertex-supported polynomial for triangle Newton polygon}

Let $g(z,w)=a+bz^3+cw$, where $a,b,c\neq0$.  Dividing by $a$ and applying
nonzero coordinate scalings, one can reduce $g=0$ to
$1+u^3+v=0$.  More precisely, choose $\lambda,\mu\in\C^\ast$ satisfying
$(b/a)\lambda^3=1$ and $(c/a)\mu=1$, and put $z=\lambda u$ and
$w=\mu v$.  Multiplication of coordinates by nonzero constants translates
the amoeba and therefore preserves the number and local types of contour
singularities.

Now put $U=u^3$ and $V=v$.  The monomial map
$\Phi(u,v)=(u^3,v)$ is a finite unramified covering of the algebraic
torus.  In logarithmic coordinates it satisfies
$$
\Log\Phi(u,v)=
\begin{pmatrix}3&0\\0&1\end{pmatrix}
\Log(u,v).
$$
The matrix is invertible over $\R$.  Hence the contour of
$1+u^3+v=0$ is obtained from the contour of $1+U+V=0$ by applying the
invertible linear map $(X,Y)\mapsto(X/3,Y)$.  The several points in a
fiber of $\Phi$ have the same logarithmic image, but they do not create
different image branches: locally $\Phi$ is a biholomorphism and all
these lifts project to the same transformed branch.

Since the contour of $1+U+V=0$ is an embedded real-analytic curve, its
image under an invertible linear map is also embedded and real analytic.
It follows that
$
\#\Sing_{\mathrm{isol}}\calC(\calA_g)=0
$
for every polynomial $g$ supported exactly on the three vertices of
$\Delta$.

\subsection*{A polynomial with the same polygon and a nodal contour}

Consider
$
f(z,w)=w-1+z-z^3.
$
Its exponent set is
$\{(0,1),(0,0),(1,0),(3,0)\}$, and therefore
$\Newt(f)=\Delta$.  The exponent $(1,0)$ is a nonvertex lattice point.
The curve $V_f$ is smooth in $\T$ because $f_w=1$ everywhere.

Write $w=p(z)$, where $p(z)=1-z+z^3$.  The two distinct points
$
p_+=(1,1),\, p_-=(-1,1)
$
belong to $V_f$, because $p(1)=p(-1)=1$.  Their logarithmic images
coincide:
$
\Log(p_+)=\Log(p_-)=(0,0).
$

The logarithmic Gauss map of $V_f$ is
$\gamma_f(z,w)=[zf_z:wf_w]=[z(1-3z^2):w]$.  At the two points one has
$$
\gamma_f(p_+)=[-2:1],\qquad
\gamma_f(p_-)=[2:1].
$$
Both values belong to $\mathbb P^1_\R$.  Hence $p_+$ and $p_-$ are
distinct logarithmic critical points lying over the same contour value.

It remains to prove that the two resulting contour branches are regular
and transverse.  On the curve $w=p(z)$, use the affine logarithmic Gauss
coordinate
$$
h(z)=\frac{z(1-3z^2)}{1-z+z^3}.
$$
Its numerator is $N(z)=z-3z^3$ and its denominator is
$D(z)=1-z+z^3$.  Since
$N'(z)=1-9z^2$ and $D'(z)=-1+3z^2$, direct differentiation gives
$
h'(1)=-4,\,  h'(-1)=-12.
$
In particular, $h'$ does not vanish at either point.  Therefore
$h^{-1}(\R)$ is a smooth real-analytic curve near $z=1$ and near
$z=-1$.  Because $h$ has real coefficients and maps the real axis to
$\R$, this local critical curve is precisely the real axis near each of
these two points.

Parametrize the first local critical branch by real $t$ near $1$ and the
second by real $t$ near $-1$.  Their logarithmic images are both given
locally by
$r(t)=(\log|t|,\log|p(t)|)$.  Since
$p'(t)=-1+3t^2$, the tangent vectors at the two critical preimages are
$$
r'(1)=\left(1,\frac{p'(1)}{p(1)}\right)=(1,2),
\qquad
r'(-1)=\left(-1,\frac{p'(-1)}{p(-1)}\right)=(-1,2).
$$
Their determinant is
$$
\det\begin{pmatrix}1&-1\\2&2\end{pmatrix}=4\neq0.
$$
Thus both projected branches are regular and their tangent lines are
distinct.  The common contour value $(0,0)$ is therefore an ordinary
transverse node.

\begin{theorem}
For
$\Delta=\operatorname{conv}\{(0,0),(3,0),(0,1)\}$, every polynomial
whose support is exactly $\operatorname{Vert}(\Delta)$ has a contour
without finite singularities, whereas
$f(z,w)=w-1+z-z^3$ has Newton polygon $\Delta$ and its contour has an
ordinary node at $(0,0)$.  Hence the supremum of the number of isolated
contour singularities over all polynomials with Newton polygon $\Delta$
is not attained in the vertex-supported subfamily.
\end{theorem}

\begin{proof}
The vertex-supported assertion follows from the reduction to the line
$1+U+V=0$ and the injective immersion $\rho$ proved above.  The polynomial
$f=w-1+z-z^3$ is smooth, has Newton polygon $\Delta$, and has two
distinct regular critical lifts $p_+$ and $p_-$ over $(0,0)$.  Their
projected tangent determinant equals $4$, so $(0,0)$ is an ordinary
node.  The strict inequality between the two suprema follows.
\end{proof}

 
The following amoebas and their contours are amoebas of plane curves with defining polynomials having the same Newton polygon $\Delta$ with vertices $\operatorname{Vert}(\Delta) = \{(0,0), (3, 0), (3, 1), (1, 3), (0, 3)\}$, but with different coefficients (the coefficients here are all real).

 \begin{figure}[ht]
\centering
\includegraphics[width=.35\textwidth]{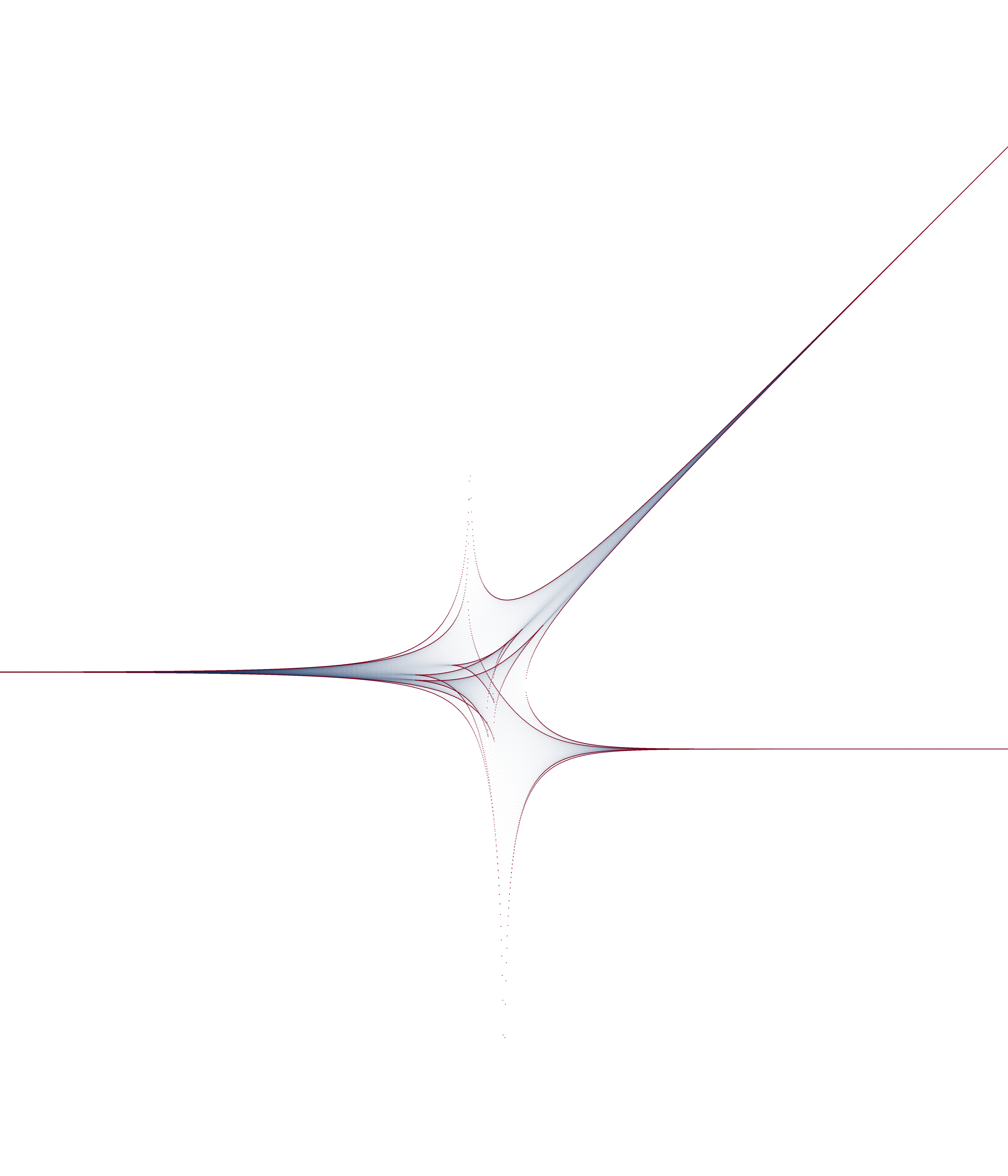}\qquad
\includegraphics[width=.35\textwidth]{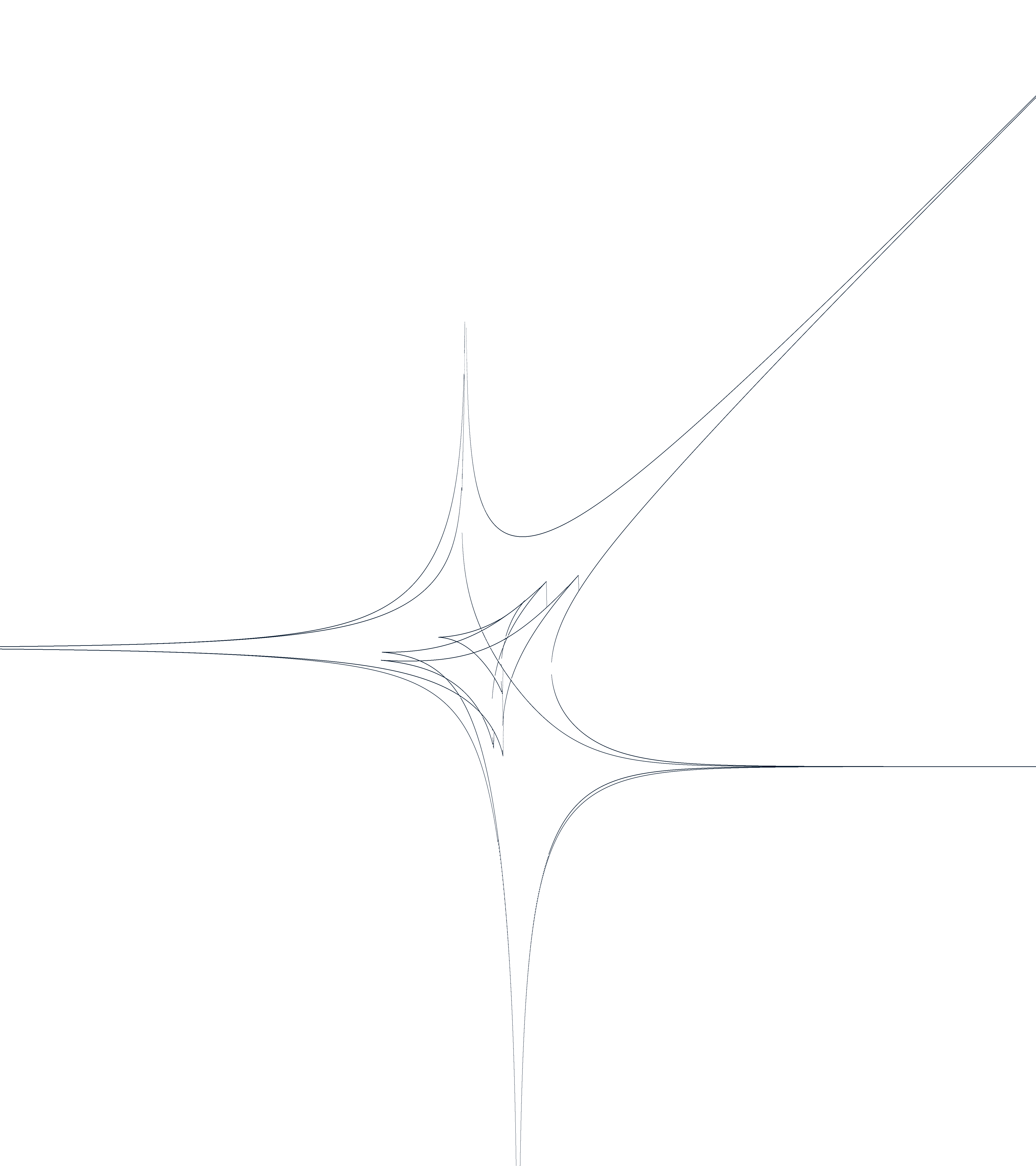}
\caption{Dark-color amoeba and clear red logarithmic critical values of the polynomial $f(z,w)=2+2z^3+2w^3-5z^3w-3zw^3+zw^2-3zw+3z^2w$.}
\end{figure}

 \begin{figure}[ht]
\centering
\includegraphics[width=.3\textwidth]{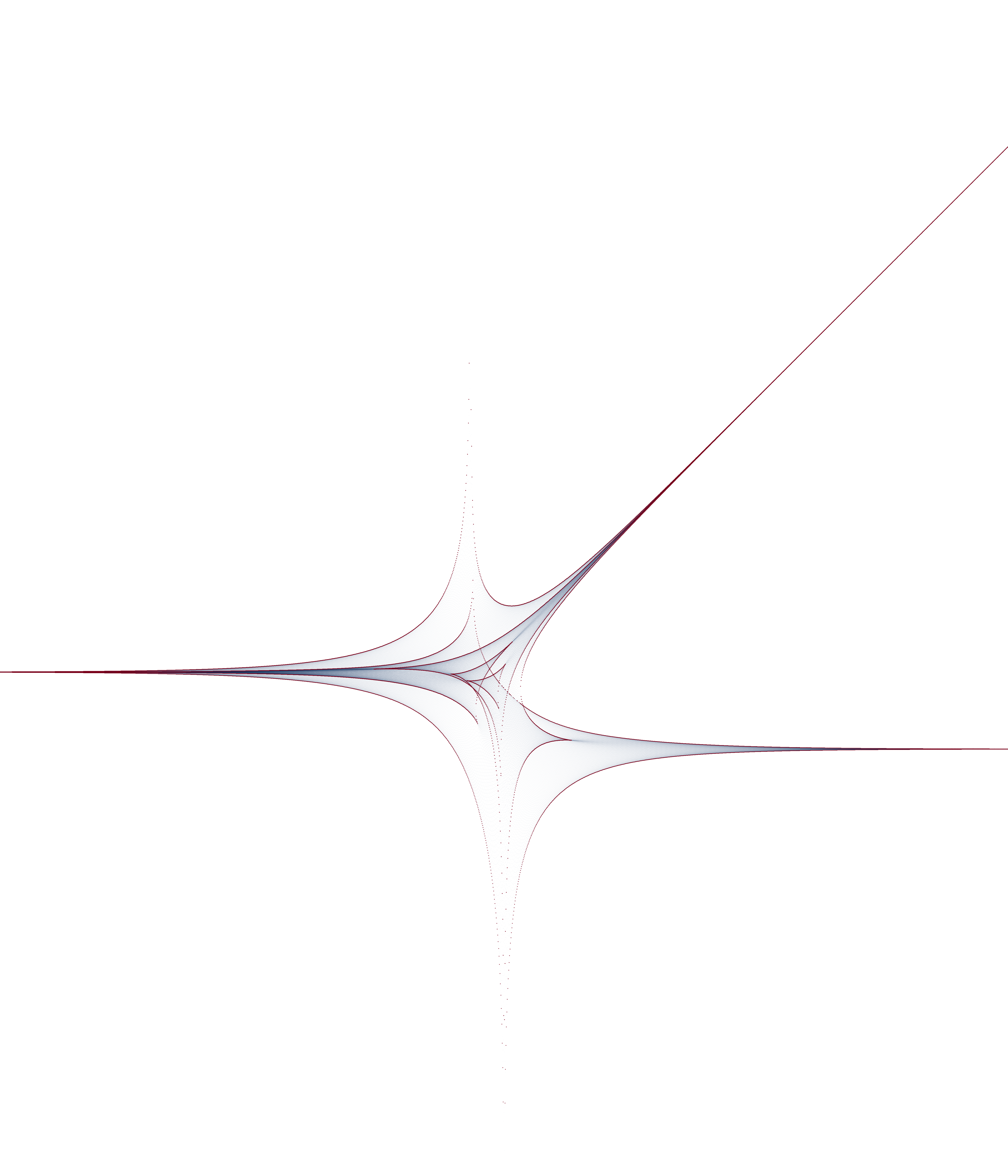}\qquad
\includegraphics[width=.3\textwidth]{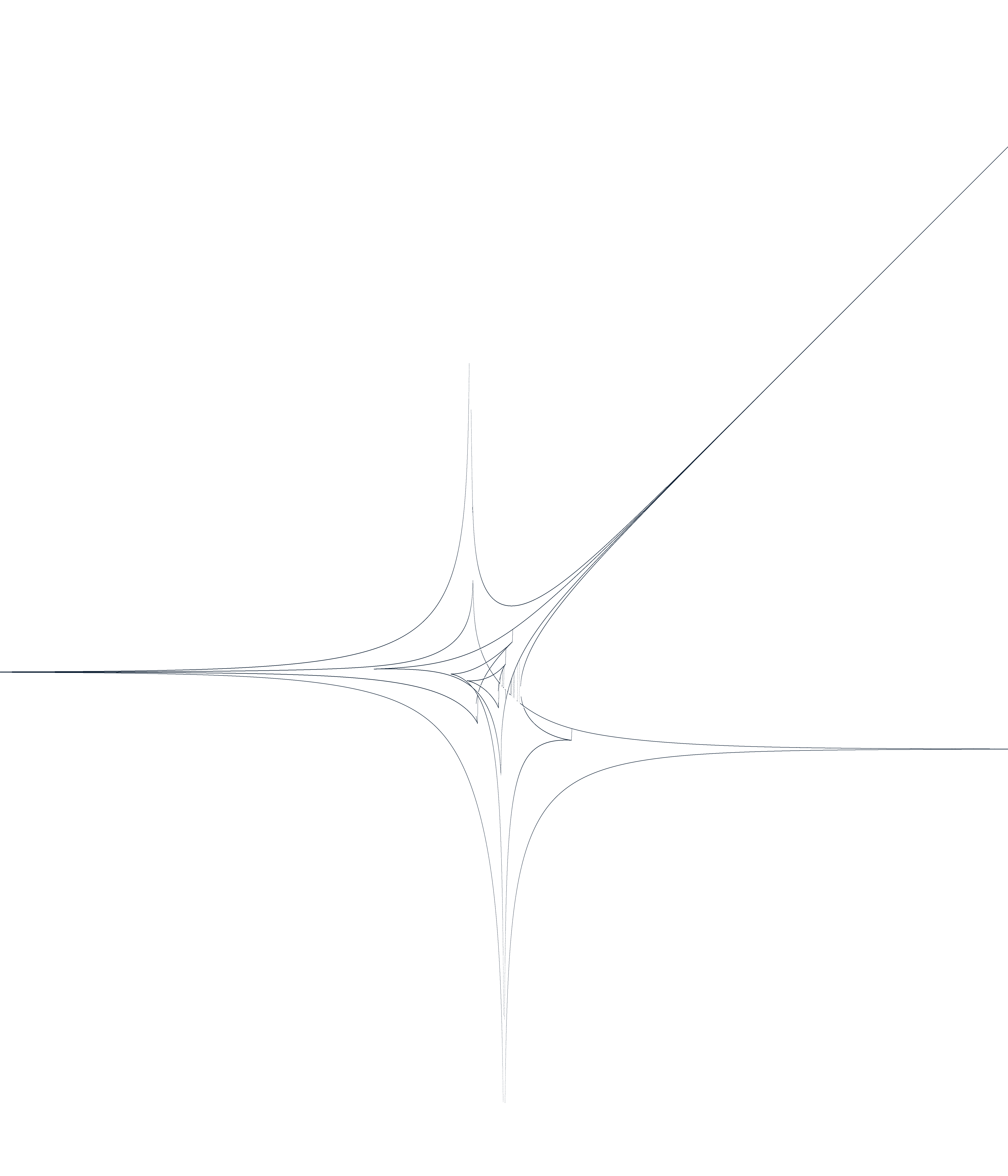}
\caption{Dark-color amoeba and clear red logarithmic critical values of the maximally spares polynomial $f(z,w)=2+2z^3+2w^3-5z^3w-3zw^3$.}
\end{figure}

\begin{figure}[ht]
\centering
\includegraphics[width=.3\textwidth]{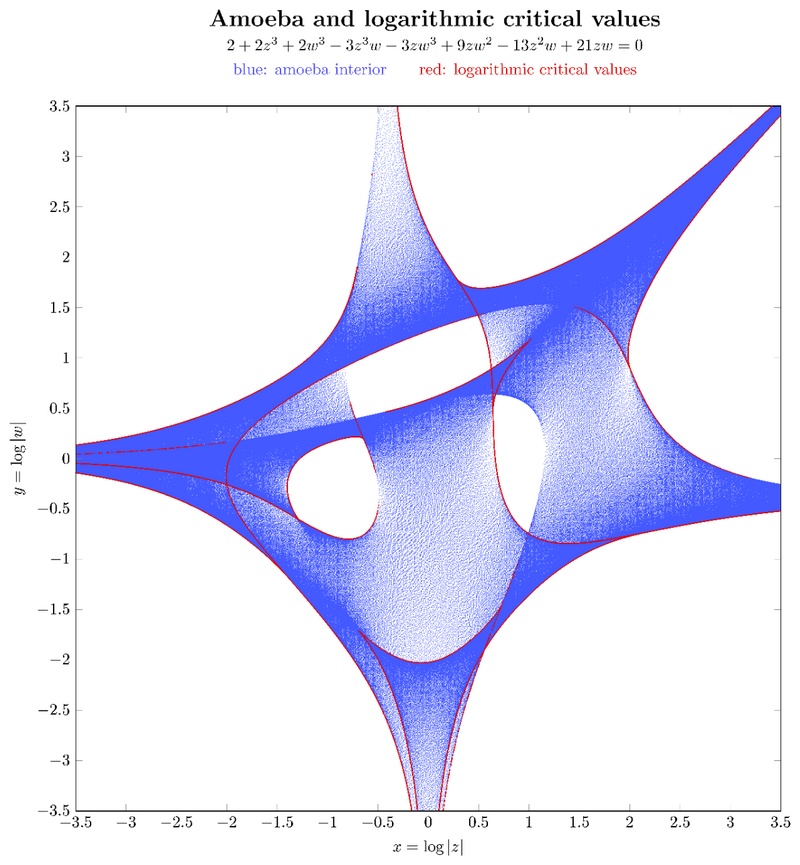}\qquad
\includegraphics[width=.3\textwidth]{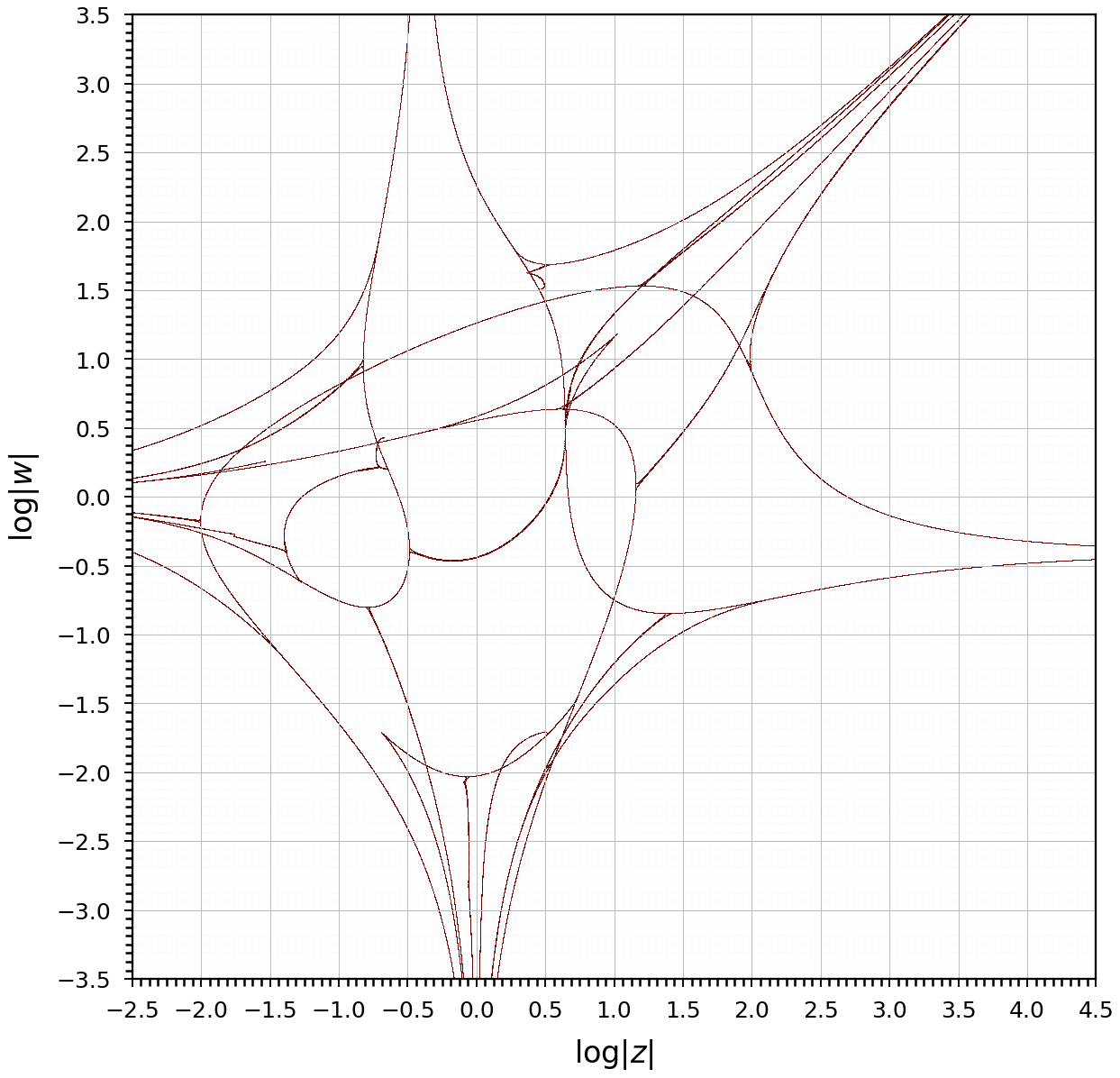}
\end{figure}

 \begin{figure}[ht]
\centering
\includegraphics[width=0.3\textwidth]{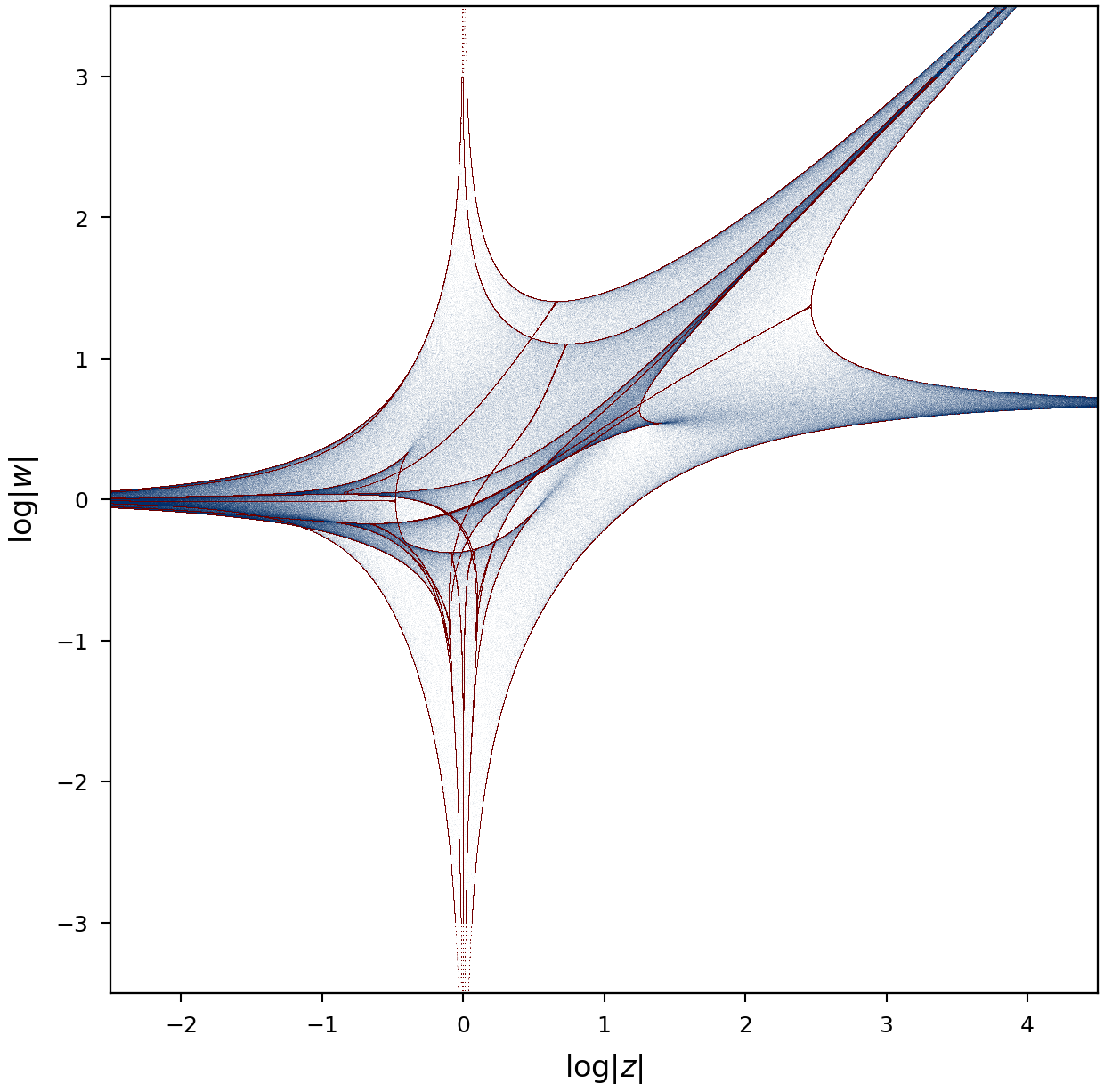}
\includegraphics[width=0.3\textwidth]{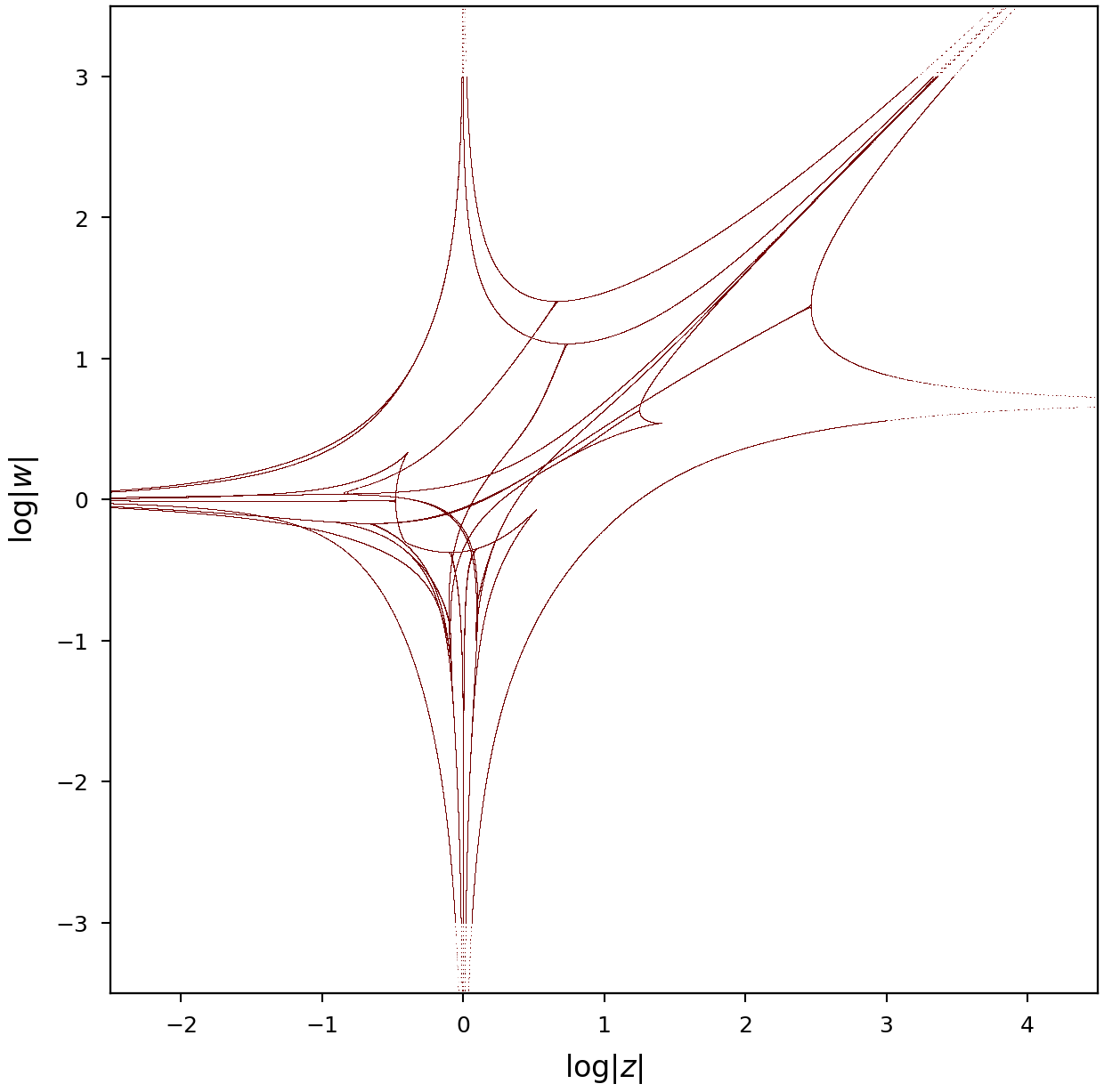}
\includegraphics[width=0.5\textwidth]{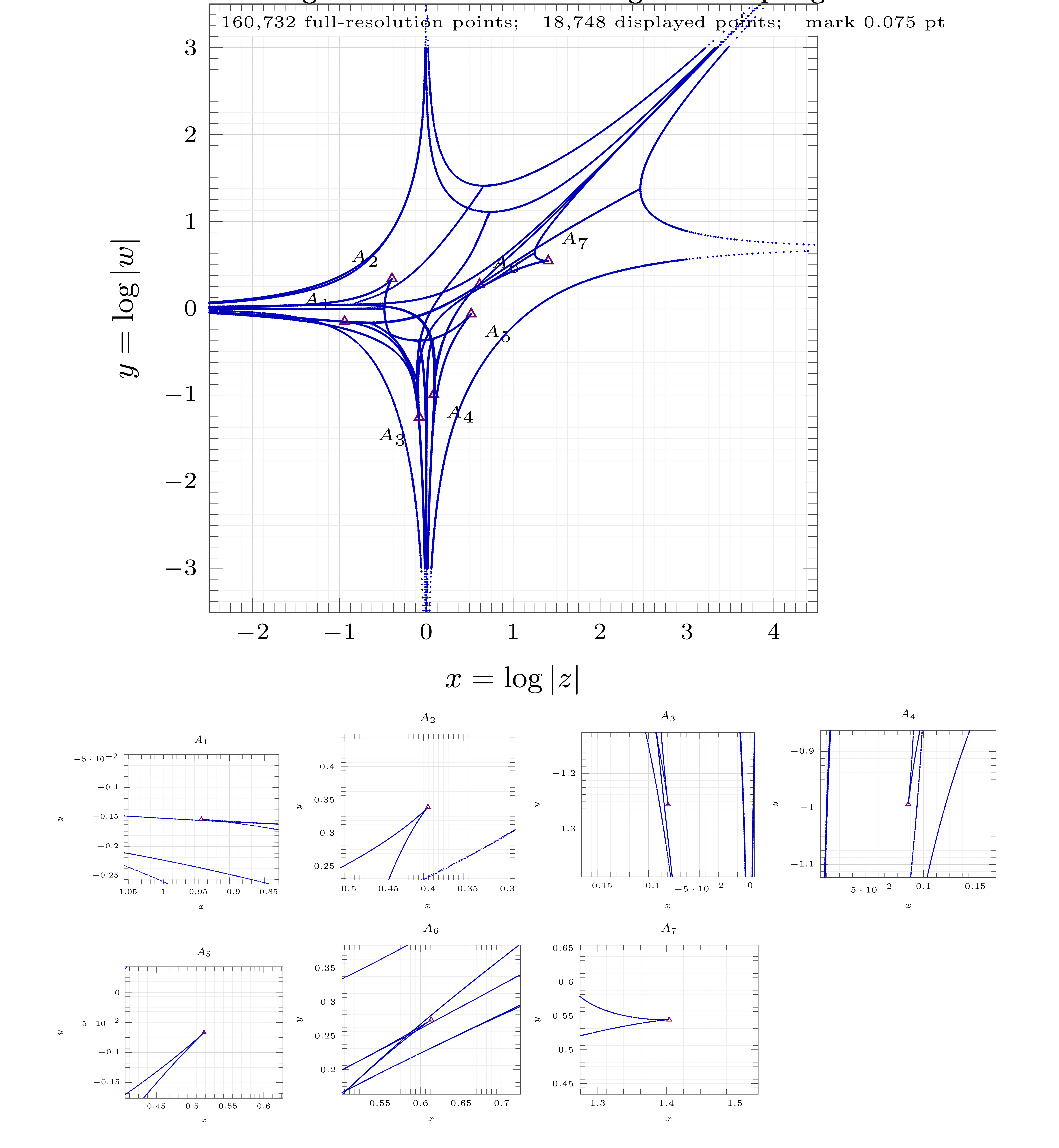}
\includegraphics[width=0.3\textwidth]{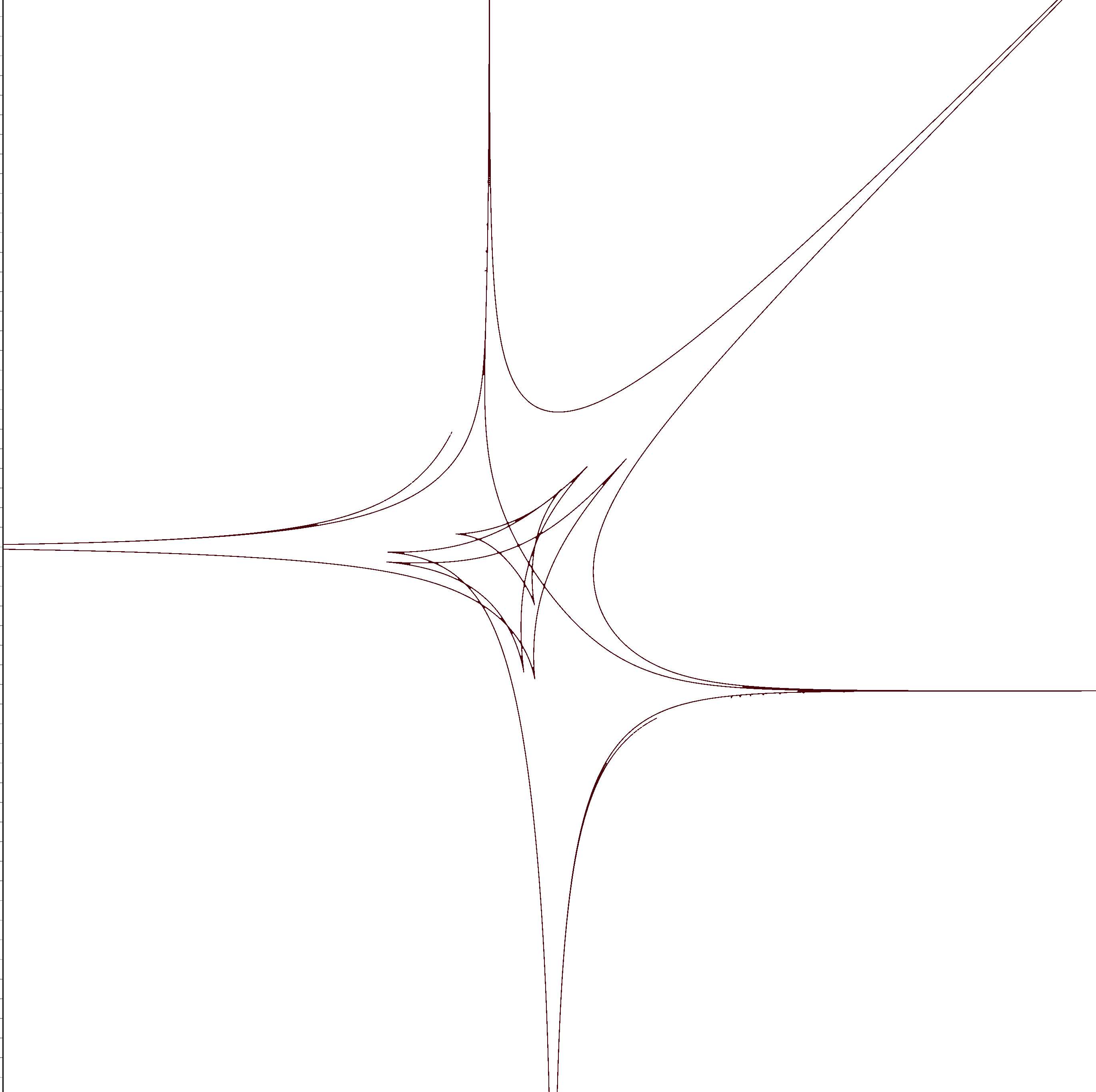}
\end{figure}

\end{document}